\documentclass[11pt, reqno]{amsart} 
\pdfoutput=1

\usepackage{amssymb}
\usepackage{mathtools}
\usepackage{mathrsfs}
\usepackage{stmaryrd}
\usepackage{upgreek}
\usepackage{centernot}
\usepackage{cancel}

\usepackage{bm}
\usepackage{aliascnt}

\usepackage[T2A, T1]{fontenc}
\usepackage[utf8]{inputenc}
\usepackage[russian, french, english]{babel}
\usepackage{lmodern}
\usepackage[spacing=true,kerning=true,babel=true,tracking=true,nopatch=footnote]{microtype}
\usepackage[shortcuts]{extdash}
\usepackage{indentfirst}
\usepackage{xspace}
\usepackage{csquotes}

\usepackage[a4paper, left=2cm, right=2cm, top=1in, bottom=1in]{geometry}

\usepackage{graphicx}
\usepackage{wrapfig}
\usepackage{float}
\usepackage{array}
\usepackage{multicol}
\usepackage{mdframed}
\usepackage[skins]{tcolorbox}
\usepackage{framed}
\usepackage[justification=centering, labelfont=bf, font=small]{caption} 

\usepackage{tikz}
\usetikzlibrary{cd}
\usetikzlibrary{shapes} 
\usetikzlibrary{arrows.meta}
\usetikzlibrary{decorations.pathmorphing} 
\usetikzlibrary{patterns}

\usepackage[
    sortcites,
    backend=biber, 
    style=alphabetic, 
    sorting=nyt, 
    maxnames=100,
    backref=false
]{biblatex}
\renewbibmacro{in:}{}
\renewbibmacro*{volume+number+eid}{%
	\printfield{volume}%
	\setunit*{\addnbspace}%
	\printfield{number}%
	\setunit{\addcomma\space}%
	\printfield{eid}}

\DeclareFieldFormat[article]{volume}{\textbf{#1}\space}
\DeclareFieldFormat[article]{number}{\mkbibparens{#1}}
\DeclareFieldFormat{journaltitle}{#1,}
\DeclareFieldFormat[thesis]{title}{\mkbibemph{#1}\addperiod}
\DeclareFieldFormat[article, unpublished, thesis, incollection, inproceedings]{title}{\mkbibemph{#1},}
\DeclareFieldFormat[book]{title}{\mkbibemph{#1}\addperiod}
\DeclareFieldFormat[unpublished]{howpublished}{#1, }
\DeclareFieldFormat{pages}{#1}
\DeclareFieldFormat[article]{series}{Ser.~#1\addcomma}

\newtheoremstyle{bfnote}%
{}{}%
{\slshape}{}%
{\bfseries}{\bfseries.}%
{ }%
{\thmname{#1}\thmnumber{ #2}\thmnote{ (\normalfont{}#3)}}
\newtheoremstyle{bfnoteupright}%
{}{}%
{\normalfont}{}%
{\bfseries}{\bfseries.}%
{ }%
{\thmname{#1}\thmnumber{ #2}\thmnote{ (\normalfont{}#3)}}

\theoremstyle{bfnote}
\newtheorem{thm}{Theorem}[section]
\newcommand{\newaliastheorem}[2]{%
    \newaliascnt{#1}{thm}%
    \newtheorem{#1}[#1]{#2}%
    \aliascntresetthe{#1}%
    \expandafter\def\csname theH#1\endcsname{\theHthm}%
}
\newaliastheorem{obs}{Observation}
\newaliastheorem{ques}{Question}
\newaliastheorem{lem}{Lemma}
\newaliastheorem{coro}{Corollary}
\newaliastheorem{cors}{Corollaries}
\newaliastheorem{prop}{Proposition}
\newaliastheorem{conj}{Conjecture}
\theoremstyle{bfnoteupright}
\newaliastheorem{exmp}{Example}
\theoremstyle{bfnote}
\theoremstyle{bfnoteupright}
\newaliastheorem{defi}{Definition}
\theoremstyle{bfnote}
\newaliastheorem{cons}{Construction}
\newaliastheorem{nota}{Notation}
\newaliastheorem{prob}{Problem}
\newaliastheorem{rem}{Remark}
\newaliastheorem{cor}{Corollary}

\newaliastheorem{pa}{Page}
\newaliastheorem{clm}{Claim}

\theoremstyle{remark}
\newtheorem*{ques*}{Question}
\newtheorem*{remk*}{Remark}

\newcommand*{\myproofname}{Proof}

\makeatletter
\newcommand{\neutralize}[1]{\expandafter\let\csname c@#1\endcsname\count@}
\makeatother

\makeatletter
\newcommand{\authorfootnote}[1]{%
  \begingroup
  \renewcommand{\@makefnmark}{}%
  \renewcommand{\@makefntext}[1]{\noindent ##1}%
  \footnotetext{#1}%
  \endgroup
}
\makeatother

\usepackage[hidelinks]{hyperref}
\usepackage{cleveref} 

\crefname{thm}{theorem}{theorems}
\Crefname{thm}{Theorem}{Theorems}

\crefname{lem}{lemma}{lemmas}
\Crefname{lem}{Lemma}{Lemmas}

\crefname{prop}{proposition}{propositions}
\Crefname{prop}{Proposition}{Propositions}

\crefname{cor}{corollary}{corollaries}
\Crefname{cor}{Corollary}{Corollaries}
\crefname{coro}{corollary}{corollaries}
\Crefname{coro}{Corollary}{Corollaries}
\crefname{cors}{corollary}{corollaries}
\Crefname{cors}{Corollary}{Corollaries}

\crefname{defi}{definition}{definitions}
\Crefname{defi}{Definition}{Definitions}

\crefname{rem}{remark}{remarks}
\Crefname{rem}{Remark}{Remarks}

\crefname{obs}{observation}{observations}
\Crefname{obs}{Observation}{Observations}

\crefname{conj}{conjecture}{conjectures}
\Crefname{conj}{Conjecture}{Conjectures}

\crefname{clm}{claim}{claims}
\Crefname{clm}{Claim}{Claims}

\crefname{exmp}{example}{examples}
\Crefname{exmp}{Example}{Examples}

\crefname{ques}{question}{questions}
\Crefname{ques}{Question}{Questions}

\crefname{cons}{construction}{constructions}
\Crefname{cons}{Construction}{Constructions}

\crefname{nota}{notation}{notations}
\Crefname{nota}{Notation}{Notations}

\crefname{prob}{problem}{problems}
\Crefname{prob}{Problem}{Problems}

\crefname{pa}{page}{pages}
\Crefname{pa}{Page}{Pages}

\crefname{fact}{fact}{facts}
\Crefname{fact}{Fact}{Facts}

\crefname{ack}{acknowledgment}{acknowledgments}
\Crefname{ack}{Acknowledgment}{Acknowledgments}

\usepackage[inline]{enumitem}
\setlist{topsep=3pt,itemsep=3pt} 

\usepackage{setspace}
\newcommand{\B}{\mathcal{B}}
\newcommand{\G}{\mathcal{G}}
\newcommand{\F}{\mathcal{F}}
\newcommand{\M}{\mathcal{M}}
\newcommand{\T}{\mathcal{T}}
\renewcommand{\H}{\mathcal{H}}
\newcommand{\cL}{\mathcal{L}}
\newcommand{\R}{\mathbb{R}}
\newcommand{\cB}{\mathcal{B}}

\newcommand{\cS}{\mathcal{S}}
\newcommand{\cH}{\mathcal{H}}
\newcommand{\cX}{\mathcal{X}}
\newcommand{\dt}{\mathrm{d} t}

\newcommand{\dtv}{d_{\mathrm{tv}}}
\newcommand{\mB}{\mathcal{M}(\B)}
\newcommand{\mpB}{\mathcal{M}_+(\B)}
\newcommand{\mBB}{\mathcal{M}(\B^2)}
\newcommand{\mpBB}{\mathcal{M}_+(\B^2)}

\newcommand{\emphd}[1]{{\fontseries{b}\selectfont\textsf{#1}}}
\newcommand{\kab}{\kappa_{\alpha,\beta}}
\DeclareMathOperator{\bmm}{\mathsf{bmm}}

\DeclareMathOperator{\bd}{\mathsf{Bd}}

\title{\sffamily Submodular Flows and Extreme Flows on Measurable Spaces} 
\date{}

\author[J. Yu]{Jing~Yu}
\author[J. Zhang]{Junchi~Zhang}
\author[M. Zhou]{Mingyang~Zhou}

\begin{document}

\begin{abstract}
The theory of submodular flows, introduced by Edmonds and Giles, is a cornerstone of combinatorial optimization, unifying network flows, matroid intersections and directed cut coverings. In this paper, we establish a measurable-space version of this framework, addressing the structural existence and duality questions raised as part of Problem~10.6 by Lovász in \cite{lovasz2023submodular}.

We develop a theory of submodular flows on standard Borel spaces and establish the measurable analogues of the existence and optimality theorems. Furthermore, we introduce a measurable notion of the residual graph and characterize extreme flows by combining a base-polytope intersection condition with an acyclicity condition for the measurable residual graph, 
generalizing the discrete geometric intuition to the infinite-dimensional setting. Finally, 
we apply the theory to constrained supply-demand problems on measurable bipartite graphs and to fractional measurable orientations.
\end{abstract}

\maketitle
\authorfootnote{(JY) Shanghai Center for Mathematical Sciences, Fudan University, Shanghai, China; 
e-mail: \texttt{jyu@fudan.edu.cn}. Partially supported by the National Natural Science Foundation of China grant 12371343 and 12525110 (PI: Hehui Wu).\\[2pt]
(JZ) Shanghai Center for Mathematical Sciences, Fudan University, Shanghai, China; 
e-mail: \texttt{jczhang24@m.fudan.edu.cn}.\\[2pt]
(MZ) School of Mathematical Sciences, Fudan University, Shanghai, China; 
e-mail: \texttt{myzhou25@m.fudan.edu.cn}.}

\section{Introduction}

Flow problems form one of the central paradigms of combinatorial optimization. In a finite directed graph, classical results such as Hoffman's Circulation Theorem \cite{Hoffman1960flow} and the Max-Flow-Min-Cut Theorem \cite{Fulkerson1956} characterize the existence of feasible circulations and flows under lower and upper capacity constraints. A far-reaching extension of this framework was the \emph{submodular flow problem} introduced by Edmonds and Giles~\cite{Edmond1975}: Instead of prescribing only edge capacities, the boundary of the flow on every vertex set is controlled by a submodular set function. If $G=(V,E)$ is a finite digraph and $f\colon E\to \mathbb R$ is a flow, then for a subset $X\subseteq V$ one writes $\partial f(X)$ for the net outflow of $X$; a \emphd{submodular flow} is then a feasible flow $f$ satisfying 
\[
\partial f(X)\le \varphi(X)
\]
for all $X\subseteq V$,
where $\varphi\colon 2^V\to \mathbb R$ is a submodular function. This point of view unifies ordinary network flows, cut covering problems, and several polymatroidal constructions, and leads to powerful min-max theorems and polyhedral duality; see, for example,~\cite{fujishige2005submodular,frank2011connections}.

A different direction, initiated by Lov\'asz~\cite{lovasz2021flows}, extends flow theory from finite graphs to measures on the square of a measurable space $J$. In that setting, a flow is a finite signed measure on $J\times J$, a circulation is a measure with equal marginals, and lower and upper capacities are given by bounding measures. Working on standard Borel spaces, Lov\'asz established measurable analogues of Hoffman's theorem, the supply-demand theorem, and the basic optimality criterion for circulations with cost. Independently, in his study of submodular set functions on general set algebras, Lov\'asz asked whether a corresponding theory of submodular flows exists in the measurable setting~\cite[Problem~10.6]{lovasz2023submodular}, which is the main purpose of this paper.

Following~\cite{lovasz2021flows}, we occasionally use the phrase ``measurable space'' informally, but unless explicitly stated otherwise all spaces in this paper are \emph{standard Borel spaces}. Our basic object is the following. Let $(J,\B)$ be a standard Borel space, let $\alpha,\beta\in \mathcal M(\B^2)$ with $\alpha\le \beta$, and let $\varphi\colon \B\to \mathbb R$ be a bounded continuous submodular set function with $\varphi(\emptyset)=\varphi(J)=0$. For a flow $\mu\in \mathcal M(\B^2)$, its boundary is defined as
\[
\partial\mu(X):=\mu(X\times X^c)-\mu(X^c\times X),\qquad X\in\B.
\]
We call $\mu$ a \emphd{feasible submodular flow} for $(J,\B,\alpha,\beta,\varphi)$ if $\alpha\le \mu\le \beta$ and $\partial\mu\le \varphi$ on $\B$.

Our first main result is a measurable analogue of the Edmonds-Giles feasibility theorem \cite{Edmond1975}. We prove in \Cref{submoexist} that a feasible submodular flow exists if and only if
\[
\alpha(X\times X^c)-\beta(X^c\times X)\le \varphi(X)
\]
for all $X\in\B$.
Next, a submodular-constraint version of the optimality theorem for measurable circulations with cost is given in \Cref{thm:minicostsubmo}, using a separation argument and the potential representation theorem for measurable circulations.
These two theorems extend Theorem 4.4 and Theorem 4.6 in \cite{lovasz2021flows}.

Another topic of the paper is \emph{extreme feasible flows}.
For finite cases, the Birkhoff–von Neumann Theorem \cite{Birkhoff1946extreme}
describes a criterion for extreme points of the convex polytope formed by doubly stochastic matrices. 
A number of results follow in characterizing extreme measures with given marginals \cite{Diego1972extreme,Ma2025},
especially for the extreme coupling measures of Lebesgue measure on $[0,1]$; see, for example, \cite{Lindenstrauss1965extreme,Losrt1982stochastic}.
In \cite{benes1987extreme,Kevin1995extreme,Abbas2016extreme}, a necessary condition related to acyclicity for the support of extreme doubly stochastic measures is given. Motivated by the acyclicity criterion for residual graphs in the finite setting, we define the measurable residual graphs and give a characterization of extreme feasible submodular flows in \Cref{thm:extremesubmodularflow} and also of \emph{reduced} feasible submodular flows in \Cref{thm:reducedextremesubmodularflow}.

As an application, we study transshipment problems on standard Borel spaces. We use \Cref{thm:extremesubmodularflow} to characterize extreme feasible transshipments, and regarding the existence of feasible transshipments, a Hall-type criterion is given in \Cref{prop:hall}, which can be viewed as a constrained version of \cite[Proposition 4.15]{lovasz2021flows}.
As another application, \Cref{thm:measurable_orientation} gives a fractional measurable analogue of Frank’s orientation theorem \cite{frank1996orientation} for highly edge-connected graphs.
However, obtaining a non-fractional measurable orientation is more delicate, and we pose it as an open problem.

To summarize, we establish a submodular flow theory in measurable spaces and extend the existence and optimality criteria of feasible flows in \cite{lovasz2021flows} to submodular flows. Further, we give characterizations of extreme submodular flows and extreme reduced flows.

The organization of the paper is as follows.
In \Cref{sec2} we review the measure-theoretic flow framework of~\cite{lovasz2021flows}, recall the relevant facts on bounded continuous submodular set functions, and prove auxiliary results on basic minorizing charges. In \Cref{sec3} we establish the existence and optimality theorems for measurable submodular flows. In \Cref{sec4} we introduce measurable residual graphs and develop the extremality theory for submodular flows and reduced flows. Finally, \Cref{sec5} is devoted to applications in bipartite transshipment problems and measurable orientations.

\section{Preliminaries}\label{sec2}
\subsection{Measure spaces and measurable flows}\label{lovaszdefi}

Most definitions in this section are adapted from~\cite{lovasz2021flows}. 
\begin{defi}[Measure spaces]
A \emphd{standard Borel space} is a measurable space isomorphic to the Borel $\sigma$-algebra of a
Polish space. Throughout the paper we identify it with a Polish realization $(J,\B)$.
   The \emphd{diagonal} $$\Delta_J:=\{(x,x)\in J\times J\;:\; x\in J\}$$ belongs to $\B^2$.
    We denote by $\mB$ the set of all finite signed measures on $(J,\B)$ and by $\mpB$ the set of all finite measures on $(J,\B)$.
    
    For any $\alpha\in \mB$, the \emphd{total variation norm} of $\alpha$ is    $$ \| \alpha \| :=\sup_{A \in \B} \alpha(A)-\inf_{B \in \B}\alpha(B). $$
    Then $\|\cdot\|$ is a norm on $\mB$ turning it into a Banach space, with $\mpB$ a closed cone of $\mB$.
    This norm induces a metric on $\mpB$ called \emphd{total variation distance}, denoted by $\dtv(\alpha,\beta)=\| \alpha-\beta \|.$
    Since $(J\times J,\B^2)$ is also a standard Borel space, we define $\mathcal{M}(\B^2)$ and $\mpBB$ analogously.
    
    For any $A\in \B^2$ and any $\mu\in \mBB$, we denote by $\mu|_A\in \mBB$  the \emphd{restriction measure} such that $$\mu|_A(B)=\mu(A\cap B)$$ for any $B\in \B^2$.
    The \emphd{transpose} of $E\in \B^2$ is the measurable subset $$E^*:=\{(x,y)\;:\; (y,x)\in E\}.$$
    The \emphd{transpose} of $\mu\in \mBB$ is the unique signed measure $\mu^*$  satisfying $\mu^*(E)=\mu(E^*)$ for any $E\in \B^2$. We write $\mu^1,\mu^2$ for the marginals, i.e., the signed measures on $(J,\B)$ such that for any $X\in\B$, $$\mu^1(X)=\mu(X\times J), \qquad \mu^2(X)=\mu(J\times X).$$

    The order on $\mB$ is the usual measure order: $\mu\le \nu$ means $\nu-\mu\in \mpB$.

    By the \emphd{Hahn decomposition} given a standard Borel space $(J,\B)$ and $\mu\in\mB$, there exist disjoint sets $P,N\in\B$, with $P \cup N =J$ such that $\mu|_P,-\mu|_N\in \mpB$. We write 
    $$\mu_+:=\mu|_P, \qquad \mu_-:=\mu|_N.$$
    Then
    $$\mu=\mu_+-\mu_-, \qquad |\mu|=\mu_++\mu_-.$$
    Here $|\mu|$ is the \emphd{total variation measure} of $\mu$.
    
    For any two $\mu,\nu\in\mB$, the \emphd{join and meet} are defined by
    \begin{align*}
        \mu\vee\nu(X)&:=\sup\{ \mu(X_1)+\nu(X_2)\;:\; X=X_1\sqcup X_2, \;X_1,X_2\in\B\},\\
        \mu\wedge\nu(X)&:=\inf\;\{ \mu(X_1)+\nu(X_2)\;:\; X=X_1\sqcup X_2, \;X_1,X_2\in\B\}.
    \end{align*}
    It can be shown that $\mu\wedge\nu,\mu\vee\nu\in\mB$ and if $\mu,\nu\in\mpB$, then $\mu\wedge\nu,\mu\vee\nu\in\mpB$.

    For a signed measure $\nu$, we call $\nu$ \emphd{absolutely continuous} with respect to a positive measure $\mu$, denoted by $\nu\ll\mu$, if $\mu(X)=0$ implies $\nu(X)=0$ for any $X\in\B$. A measure $\alpha$ is \emphd{supported on} a measurable set $S\in\B$ if $\alpha(S^c)=0$.

    Given two signed measures $\lambda,\mu \in \mB$, we define the \emphd{product measure}, denoted by $\lambda\otimes \mu\in\mBB$, to be
    $\lambda\otimes \mu (X\times Y)=\lambda(X)\times \mu(Y)$ for any $X,Y\in \B$.

    For any $\alpha\in \mB$ and any bounded measurable function $v$ on $(J,\B)$, the integral of $v$ with respect to $\alpha$ is denoted as $$\alpha(v):=\int_{J}v(x)\,\mathrm{d}\alpha(x).$$
\end{defi}

The following is a compactness result about the space of measures.

\begin{lem}[{\cite[Lemma~3.1]{lovasz2021flows}}]\label{compactmeasure}
Let $(J,\mathcal B)$ be a standard Borel space and let $\psi\in \M_+(\B^2)$. Then every sequence $(\mu_n)_{n\ge 1}$ in
\[
K_\psi:=\{\mu\in \M(\B^2): |\mu|\le \psi\}
\]
admits a subsequence $(\mu_{n_k})_{k\ge1}$ and a measure $\mu\in K_\psi$ such that
\[
\mu_{n_k}(A)\to \mu(A)\qquad\text{for every }A\in\B^2.
\]
\end{lem}

Next we introduce the flow theory on measurable spaces.

\begin{defi}[Measurable flows, circulations and potentials]
Given a standard Borel space $(J,\B)$, a \emphd{flow} on $J$ is a finite signed measure $\mu\in\mBB$. Given $\sigma,\tau\in\mB$, a flow $\mu$ is called a $\sigma$-$\tau$ flow if $$\mu^1-\mu^2=\sigma-\tau.$$
A \emphd{circulation} is a flow $\alpha\in\mBB$ with $\alpha^1=\alpha^2$.
Note that a circulation is a $0$-$0$ flow.
A \emphd{network system} is a quadruple  $(J,\B,\alpha,\beta)$ with $\alpha, \beta \in\mBB$. A flow $\mu$ is \emphd{feasible} if $\alpha\leq \mu\leq \beta$. 

A measurable function $F \colon J \times J \to \mathbb{R}$ is called a \emphd{potential} if there exists a bounded measurable function $f \colon J \to \R$ such that for any $x,y\in J$, we have $$F(x,y) = f(x) - f(y).$$ Since adding a constant to $f$ does not change the difference, one may assume that $f$ is nonnegative.
In particular, we can assume there exists $D\in\R_+$ such that $0\leq f<D$.
\end{defi}

\begin{rem}\label{potentialint}
 Every potential $F(x,y)=f(x)-f(y)$ with $0\leq f<D$ can be expressed as 
$$ F(x,y) = \int_{0}^D \big( \mathbf{1}_{A_t}(x) - \mathbf{1}_{A_t}(y) \big) \dt,$$
where $A_t= \{ x \in J : f(x) \geq t \}$, $0 \leq t \leq D$, is a measurable subset of $J$ such that $$A_t \supseteq A_s,\;\forall\;t < s;\;\; A_D = \emptyset,\;\;A_0 = J.$$    
\end{rem}

The existence criterion of feasible flows is given as follows.

\begin{thm}[{\cite[Theorem 4.4]{lovasz2021flows}}]\label{circuexistence}
    Given a standard Borel space $(J,\B)$ and $\alpha,\beta\in\mBB$, there exists a circulation $\mu$ such that $\alpha\leq \mu\leq \beta$ if and only if $\alpha\leq \beta$ and 
    \[\alpha(X\times X^c)\leq \beta(X^c\times X) \quad \text{ for all $X\in\B$}.\]
\end{thm}

\begin{coro}\label{flowexistence}
    Given a standard Borel space $(J,\B)$, two signed measures $\sigma,\tau\in\mB$ with $\sigma(J)=\tau(J)$ and two signed measures $\alpha,\beta\in\mBB$ with $\alpha \le \beta$, there exists a $\sigma$-$\tau$ flow $\mu$ such that $\alpha\leq \mu\leq \beta$ if and only if  
    $$\alpha(X\times X^c)-\beta(X^c\times X)\leq \sigma(X)-\tau(X) \quad \text{for all $X\in\B$}.$$
\end{coro}

\begin{proof}
    Define $J':=J\cup \{s,t\}$ and let $\B'$ be the $\sigma$-algebra generated by $\B \cup \{\{s\},\{t\}\}$. Let 
\begin{align*}
    \alpha'(E)\; (\text{w.r.t. }\beta'(E)):=\left\{
    \begin{aligned}
        &\alpha(E)&& (\text{w.r.t. }\beta(E)),\;& &E\subseteq J\times J;\\
        &\sigma(X)&& (\text{w.r.t. }\sigma(X)),\; & &E=\{s\}\times X\subseteq\{s\}\times J;\\
        &\tau(X)&& (\text{w.r.t. }\tau(X)),\;& &E=X\times \{t\}\subseteq J\times \{t\};\\
        &0&& (\text{w.r.t. }0),\; && E\subseteq (J\times \{s\})\cup (\{t\}\times J)\cup( \{s\}\times\{t\});\\
        &\sigma(J)=\tau(J)&& (\text{w.r.t. }\sigma(J)=\tau(J)),\;&& E=\{t\}\times \{s\},        
    \end{aligned} 
    \right.
\end{align*}
and extend them to ${\B'}^2$. Since $\alpha \le \beta$, we have $\alpha'\le \beta'$. 
Apply~\Cref{circuexistence} with $J'$, $\B'$, $\alpha'$ and $\beta'$ in place of $J$, $\B$, $\alpha$ and $\beta$, such flow exists if and only if for any $X\in \B'$, we have $\alpha'(X\times X^c)\leq \beta'(X^c\times X)$.
Since $\sigma(J)=\tau(J)$, it is easy to check that this is equivalent to saying for any $X\subseteq J$, we have
$$\phantom{ssssssssssssssssssssssssss}\alpha(X\times X^c)-\beta(X^c\times X)\leq \sigma(X)-\tau(X).\phantom{ssssssssssssssssssssssssss}\qedhere$$
\end{proof}

Two extra lemmas are needed.

\begin{lem}[{\cite[Lemma 3.4]{lovasz2021flows}}]
    \label{thm:separatinghyperplane}
    Let $K_{1},\dots,K_{n}$ be open convex sets in a Banach space $(X, \| \cdot \|)$. Then $K_{1}\cap \cdots \cap K_{n}=\emptyset$ if and only if there are bounded linear functionals $\mathcal{L}_{1},\dots,\mathcal{L}_{n}$ on $X$ and real numbers $a_{1},\dots,a_{n}$ such that $\mathcal{L}_{1}+\cdots +\mathcal{L}_{n}=0$, $a_{1}+\cdots+a_{n}=0$, and for each $i$, either $\mathcal{L}_{i}=0$ and $a_{i}=0$, or $\mathcal{L}_{i}(x)>a_{i}$ for $x \in K_{i}$, and for at least one $i$ the second possibility holds.
\end{lem}

\begin{lem}[{\cite[Lemmas~4.2 and 4.3]{lovasz2021flows}}]\label{convexseparation}
Given a standard Borel space $(J,\B)$ and a flow $\alpha\in \mBB$, then $\alpha$ is a circulation if and only if $\alpha(F)=0$ for any potential $F$ on $J\times J$.
Further, given a measure $\psi\in\mpBB$ and a bounded linear functional $\cL$ on the linear subspace $\{\mu\in\mBB\;:\: \mu\ll \psi \}$ with $\cL(c)=0$ for all circulations $c\ll \psi$, there exists a potential $F$ such that $\cL(\mu)=\mu(F)$ for all $\mu\ll\psi$.
 \end{lem}

\subsection{Submodular set functions}\label{Submodularpre}
\begin{defi}
    A \emphd{lattice} is a pair $(J,\B)$ such that $J$ is a set and $\B$ is a family of subsets of $J$ closed under finite union and intersection. 
    A \emphd{set algebra} is a lattice containing $J$ and closed under taking complement.
    A \emphd{submodular set function} is a function $f\colon \B\to\R$ such that for all $A,B\in \B$,  $$f(A)+f(B)\geq f(A\cap B)+f(A\cup B).$$
    It is \emphd{supermodular} if the inequality is reversed, and \emphd{modular} if the equality holds.
    A \emphd{charge} is a finitely additive set function. Note that a modular set function $\mu$ is a charge if and only if $\mu(\emptyset)=0$.
    A set function $\varphi$ is called \emphd{continuous} if for any $A\in \B$ and any sequence $\{A_n\}_{n=1}^\infty\subseteq \B$ such that $\lim_{n\to\infty}\mathbf{1}_{A_n}(x)=\mathbf{1}_A(x)$ for any $x\in J$, we have $\lim_{n\to\infty}\varphi(A_n)=\varphi(A)$.
\end{defi}

In this paper, we focus on bounded continuous submodular set functions $\varphi$ on a standard Borel space $(J,\B)$ with $\varphi(\emptyset)=0$; in most applications we also assume $\varphi(J)=0$.

\begin{defi}
     Let $\varphi$ be a submodular function on a standard Borel space $(J,\B)$ with $\varphi(\emptyset)=0$. The set of \emphd{basic minorizing charges}, denoted by $\bmm(\varphi)$, is defined by
     $$\bmm(\varphi):=\{\beta\;:\;\text{$\beta$ is a charge, } \beta\leq \varphi, \beta(J)=\varphi(J)\},$$
     and $\bmm_+(\varphi)$ is defined by
     $$\bmm_+(\varphi):=\{\beta\;:\;\text{$\beta$ is a charge, }0\leq \beta\leq \varphi, \beta(J)=\varphi(J)\}.$$
\end{defi}

The following \Cref{lem:countablyadditive} is folklore and similar statements can be found in~\cite{gilboa2004uncertainty}. For the sake of completeness we give a proof. We first introduce two properties. For a subset $K\subseteq \mBB$, $K$ is called \emphd{uniformly countably additive} if for any decreasing sequence $\{E_i\}_{i=1}^\infty \subset \B$ with $\bigcap_{i=1}^\infty E_i=\emptyset$ and any $\epsilon>0$, there exists $N_\epsilon$ such that for any $n\geq N_\epsilon$ and $\mu\in K$, we have $|\mu(E_n)|<\epsilon$; $K$ is called \emphd{uniformly absolutely continuous} with respect to some $\lambda\in \mpBB$ if for any $\epsilon>0$ there exists $\delta>0$ such that for any $E\in \B$ with $\lambda(E)<\delta$ we have $\sup_{\mu\in K}|\mu(E)|<\epsilon$. These two properties are equivalent as stated below.

\begin{thm}[{\cite[IV, Chapter 9, Theorems~1 and 2]{Nelson1988Linear}}]\label{thm:bmmweakcompact}
    Given a standard Borel space $(J,\B)$, let $K$ be a bounded subset of $\mBB$ in the total variation norm. Then $K$ is uniformly countably additive if and only if there exists $\lambda\in\mpBB$ such that $K$ is uniformly absolutely continuous with respect to $\lambda$.
\end{thm}

Now we can state and prove the following lemma.

\begin{lem}\label{lem:countablyadditive}
Let $(J,\B)$ be a standard Borel space and let $\varphi$ be a bounded continuous submodular set function on $\B$ with $\varphi(\emptyset)=0$.
Then any charge $\alpha\in\bmm(\varphi)$ is countably additive, and hence $\alpha\in\mB$.
Moreover, there exists some $\lambda_\varphi\in \mpB$ such that $\alpha\ll\lambda_\varphi$ for any $\alpha\in \bmm(\varphi)$.
\end{lem}
\begin{proof}
    Let $X\in\B$ and $\{X_n\}_{n=1}^\infty\subseteq \B$ be pairwise disjoint with $X=\bigsqcup_{n=1}^\infty X_n$. Let $$Y_k:=\bigsqcup_{n=1}^k X_n,\qquad Z_k:=\bigsqcup_{n=k+1}^\infty X_n.$$
    Since $\B$ is a $\sigma$-algebra, $Y_k,Z_k\in\B$.
    Since $\varphi$ is continuous, we have $\lim_{k\to\infty}\varphi(Z_k)=\varphi(\emptyset)=0$ and $\lim_{k\to\infty}\varphi(Z_k^c)=\varphi(J)$. 
    
    For every $k$, finite additivity gives $\alpha(X)=\alpha(Y_k)+\alpha(Z_k)$. Moreover,
    \[
    \varphi(J)-\varphi(Z_k^c)\le \varphi(J)-\alpha(Z_k^c)=\alpha(Z_k)\le\varphi(Z_k),
    \]
    so $\lim_{k\to\infty}\alpha(Z_k)=0$. Hence if $\lim_{k\to\infty}\alpha(Y_k)$ exists, then
    \[
    \alpha(X)=\lim_{k\to\infty}\alpha(Y_k).
    \]
    
    It remains to show that this limit always exists. Suppose not. There exists $\epsilon>0$ and a subsequence $\{Y_{n_k}\}$ such that $$|\alpha(Y_{n_{2k+1}})-\alpha(Y_{n_{2k+2}})|>\epsilon, \quad k\geq 1.$$
    Let $W_k:=Y_{n_{2k+2}}\setminus Y_{n_{2k+1}}$. Then
    $\liminf_{k\to\infty} |\alpha(W_k)|>\epsilon$.
    On the other hand, $$ \varphi(J)-\varphi(W_k^c)\leq \varphi(J)-\alpha(W_k^c) = \alpha(W_k)\leq \varphi(W_k).$$
    By continuity of $\varphi$, both the left and right bounds tend to $0$, so $\lim_{k\to\infty}\alpha(Z_k)=0$, a contradiction. Thus $\alpha$ is countably additive.

    Let $M:=\sup_{X\in\B}|\varphi(X)|<\infty$. For every $X\in\B$, $\varphi(J)-\varphi(X^c)\leq \alpha(X)\leq \varphi(X)$, and hence $-2M\le\alpha(X)\le M$. Therefore
    $\|\alpha\|=\sup_X\alpha(X)-\inf_X\alpha(X)\le3M$, so $\bmm(\varphi)$ is bounded in the total variation norm.
    
    Finally, let $\{E_i\}_{i=1}^\infty \subset \B$ be decreasing such that $\bigcap_{i=1}^\infty E_i=\emptyset$ and $\epsilon>0$. Since $E_n^c\uparrow J$ pointwise, continuity of $\varphi$ gives $\varphi(E_n)\to0$ and $\varphi(E_n^c)\to\varphi(J)$. The same two-sided estimate gives
    \[
    \sup_{\alpha\in\bmm(\varphi)}|\alpha(E_n)|\to0.
    \] 

By \Cref{thm:bmmweakcompact}, there exists $\lambda_\varphi\in\mpB$ such that $\bmm(\varphi)$ is uniformly absolutely continuous with respect to $\lambda_\varphi$. In particular, every $\alpha\in\bmm(\varphi)$ satisfies $\alpha\ll\lambda_\varphi$.
\end{proof}

By \Cref{lem:countablyadditive}, if $\varphi$ is bounded and continuous, then $\bmm(\varphi), \bmm_+(\varphi) \subseteq \mB$.

We also need the following two basic properties about submodular set functions.

\begin{thm}[{\cite[Theorem 5.1]{lovasz2023submodular}}]\label{separation}
    Let $(J,\B)$ be a set algebra and $\alpha,\beta$ be bounded set functions on $(J,\B)$, where $\alpha$ is supermodular and $\beta$ is submodular. If $\alpha\leq \beta$, then there exists a modular set function $\gamma$ on $(J,\B)$ such that $\alpha\leq \gamma\leq \beta$ and $\gamma(J)=\beta(J)$.
\end{thm}

    \begin{thm}[{\cite[Theorem 5.5]{lovasz2023submodular}}]\label{strongerseparation}
    Let $\varphi$ be a bounded submodular set function on a set algebra $(J,\mathcal{B})$ such that $\varphi(\emptyset)=0$, and let $\mathcal{L}\subseteq \mathcal{B}$ be a lattice containing $\emptyset$ and $J$. Assume that $\varphi$ is modular on $\mathcal{L}$. Then there exists $\alpha \in \bmm (\varphi)$ such that $\alpha(S)=\varphi(S)$ for every $S\in \mathcal{L}$. If $\varphi$ is increasing, then $\alpha\geq 0$.
    \end{thm}

In the end, we introduce the \textit{Choquet integral} of submodular set functions, transforming set functions to operators on bounded measurable functions.

\begin{defi}
  Given a set algebra $(J,\B)$ and a bounded submodular set function $ \varphi $ on $(J,\B)$, for any function $ f \in \bd(\B) $, where $\bd(\B)$ denotes the set of all bounded $\B$-measurable functions on $J$, choose $C\in\R$ with $f+C\geq0$ and define the \emphd{Choquet integral} as
  $$\hat{\varphi}(f):= \int f \, \mathrm{d}\varphi := \int_0^\infty \varphi(\{f+C\geq t\}) \, \dt -C \varphi(J).$$
 This definition is independent of $C$ since
  \begin{align*}
      \int_0^\infty \varphi(\{f+C_1+C_2\geq t\}) \, \dt&=\int_0^{C_2} \varphi(\{f+C_1+C_2\geq t\}) \, \dt + \int_{C_2}^\infty \varphi(\{f+C_1+C_2\geq t\}) \, \dt\\
      &= \int_0^\infty \varphi(\{f+C_1\geq t\}) \, \dt+C_2\varphi(J).
  \end{align*}
  
\end{defi}

\section{Measurable submodular flows: existence and optimality}\label{sec3}

In the spirit of classical submodular flows, we give the definition of submodular flows on a standard Borel space as follows.

\begin{defi}[Boundary operator and cut function]
    Given a standard Borel space $(J,\B)$ and $\mu\in \mBB$, define the \emphd{boundary} of $\mu$ by
    \[\partial\mu(X):=\mu(X\times X^c)-\mu(X^c\times X), \quad X\in\mathcal{B}.\] 
  Then \(\partial\mu=\mu^1-\mu^2\), so \(\partial:\mBB \to \mB\) is a linear operator.
    For $\alpha$, $\beta\in \mBB$ with $\alpha\leq \beta$, the \emphd{cut function} $\kab$ is defined by \[\kab(X):=\beta(X\times X^c)-\alpha(X^c\times X), \quad X\in\mathcal{B}.\]
\end{defi}

\begin{defi}[Feasible submodular flows]
    Let $(J,\B)$ be a standard Borel space  and let $\varphi : \B \to \R$ be a bounded continuous submodular function with $\varphi(\emptyset)=\varphi(J)=0$, a \emphd{submodular flow with respect to $\varphi$} is a signed measure $\mu \in \mBB$ such that 
    \[\partial \mu(X) \leq \varphi(X) \qquad \textrm{ for all } X\in \B.\]
    Given $\alpha,\beta\in\mBB$ with $\alpha\leq \beta$, a flow $\mu$ is \emphd{feasible with respect to} $(\alpha,\beta)$ if $\alpha\leq \mu\leq \beta$.
    The tuple $(J,\B,\alpha,\beta,\varphi)$ is called a \emphd{submodular network system}. If $v\colon J\times J\to\R_{\geq 0}$ is bounded and measurable, the \emphd{cost with respect to} $v$ of a flow $\mu$ is defined as $$\mu(v)=\int_{J \times J}\,v(x)\,\mathrm{d}\mu(x).$$
    For simplicity, we just say $\mu$ is a \emphd{feasible submodular flow} if the system is clear.
\end{defi}

\subsection{Existence of feasible submodular flows}

In this subsection, given a submodular network system, we characterize the condition for the existence of a feasible submodular flow. We first give several lemmas about the cut function and base polytope intersection.

\begin{lem}\label{kabsubmo}
    Given a network system $(J,\B,\alpha,\beta)$ with $\alpha \le \beta$, the cut function $\kab$ is bounded, continuous, and submodular. Moreover,
\[
    \kappa_{\alpha,\beta}(\emptyset)=\kappa_{\alpha,\beta}(J)=0.
\]
\end{lem}

\begin{proof}
Boundedness is immediate from the finiteness of $\alpha$ and $\beta$. If $\mathbf 1_{X_n}\to\mathbf 1_X$ pointwise, then $\mathbf 1_{X_n\times X_n^c}\to\mathbf 1_{X\times X^c}$ and $\mathbf 1_{X_n^c\times X_n}\to\mathbf 1_{X^c\times X}$ pointwise on $J\times J$. Dominated convergence with respect to $|\alpha|+|\beta|$ gives $\kappa_{\alpha,\beta}(X_n)\to\kappa_{\alpha,\beta}(X)$, so $\kab$ is continuous.

For submodularity, decompose $J$ according to the four regions $X\cap Y$, $X\setminus Y$, $Y\setminus X$, and $(X\cup Y)^c$. For every finite signed measure $m$ on $J\times J$, one has
\begin{align*}
    &m(X\times X^c)+m(Y\times Y^c)-m((X\cup Y)\times (X\cup Y)^c)-m((X\cap Y)\times (X\cap Y)^c)\\
    =\;& m\big( (X\setminus Y)\times (Y\setminus X)\big)+m\big((Y\setminus X)\times (X\setminus Y)\big)\ge 0.
\end{align*}
Applying this to $m=\beta$ and, with $X,Y$ replaced by $X^c,Y^c$, to $m=\alpha$, we get
\begin{align*}
    &\kab(X)+\kab(Y)-\kab(X\cup Y)-\kab(X\cap Y)\\
=\;&(\beta-\alpha)\big((X\setminus Y)\times (Y\setminus X)\big)+(\beta-\alpha)\big((Y\setminus X)\times (X\setminus Y)\big)\ge 0,
\end{align*}
since $\beta\ge \alpha$. Hence $\kab$ is submodular.
\end{proof}

\begin{lem}\label{partialinbmm}
    Let $(J,\B,\alpha,\beta)$ be a network system with $\alpha \le \beta$. Then $$\bmm(\kab) =\{\partial\mu\;:\; \mu\in\mBB,\;  \alpha\leq \mu\leq \beta \}.$$
\end{lem}
\begin{proof}
    Let $\mu\in\mBB$ satisfy $\alpha\leq \mu\leq \beta$. Then for every  $X \in \B$, 
    we have $$\partial\mu(X)=\mu(X\times X^c)-\mu(X^c\times X)\leq \beta(X\times X^c)-\alpha(X^c\times X)=\kab(X).$$ 
    Also, $\partial\mu(J) = 0 =\kab(J)$.
    Thus $\partial \mu\in \bmm(\kab)$.

    Conversely, let $\gamma\in \bmm(\kab)$. By \Cref{kabsubmo} and \Cref{lem:countablyadditive}, $\gamma$ is a finite signed measure. Since $\gamma(J)=\kab(J)=0$, we have $\gamma(X)=-\gamma(X^c)$ for all $X \in \B$. 
    Since $\gamma \le \kab$, applying this inequality to $X^c$ gives
    \[\gamma(X^c)\le\kappa_{\alpha,\beta}(X^c)=\beta(X^c\times X)-\alpha(X\times X^c).\]
    Therefore
    \[\alpha(X\times X^c)-\beta(X^c\times X)\le -\gamma(X^c)=\gamma(X).\]
    By~\Cref{flowexistence} with $\sigma=\gamma$ and $\tau = 0$, there exists a $\gamma$-0 flow $\mu \in \mBB$ such that $\alpha\leq \mu\leq \beta$. Hence $\partial \mu = \gamma$.
\end{proof}

\begin{lem}\label{intersectoftwosubmo}
    Given a standard Borel space $(J,\B)$ and two bounded submodular set functions $\varphi,\psi$ such that $\varphi(J)=\psi(J)$ and $\varphi(\emptyset)=\psi(\emptyset)=0$, we have 
    \[\bmm(\varphi)\cap \bmm(\psi) \neq \emptyset\] if and only if
    $$\psi(J)-\psi(X^c)\leq \varphi(X) \quad\text{for any $X \in \B$.}$$
\end{lem}
\begin{proof}
    If $\mu\in\bmm(\varphi)\cap \bmm(\psi)$, then for any $X \in \B$, $$ \varphi(X)+\psi(X^c)\geq \mu(X)+\mu(X^c)=\mu(J)=\psi(J),$$
which gives the disired inequality. 

    Conversely, assume
    $\psi(J)-\psi(X^c)\leq \varphi(X)$ for all $X \in \B$ and define the supermodular set function 
    \[\phi\colon \B\to \mathbb{R}, \quad\phi(X):=\psi(J)-\psi(X^c).\] Then  $\phi\leq \varphi$, $\phi(\emptyset)=\varphi(\emptyset)=0$, and $\phi(J)=\varphi(J)=\psi(J)$.
    By~\Cref{separation}, there exists a modular set function $\mu$ such that $\phi\leq \mu\leq \varphi$, $\mu(\emptyset)=0$ and $\mu(J)=\psi(J)$.
    For any $X\in\B$, we have $$\mu(X)=\mu(J)-\mu(X^c)\leq \psi(J)-\phi(X^c)=\psi(X).$$
    Thus $\mu\in\bmm(\varphi)\cap \bmm(\psi)$.
\end{proof}

\begin{thm}[Existence Theorem]\label{submoexist}
     Given a submodular network system $(J,\B,\alpha,\beta,\varphi)$, there exists a feasible submodular flow if and only if for any $X\in\B$, we have
     $$\alpha(X\times X^c)-\beta(X^c\times X)\leq \varphi(X).$$
\end{thm}

\begin{proof}
    By~\Cref{partialinbmm}, a feasible submodular flow exists if and only if $\bmm(\varphi)\cap \bmm(\kab) \neq \emptyset$.
    By~\Cref{kabsubmo} and \Cref{intersectoftwosubmo}, this is equivalent to
    $$\kab(J)-\kab(X^c)\leq \varphi(X) \quad \text{for any $X\in\B$}.$$
    Since $\kab(J)=0$,  this is precisely
     $$\phantom{ssssssssssssssssssssssss}\alpha(X\times X^c)-\beta(X^c\times X)\leq \varphi(X)\quad \text{for any $X\in\B$}.\phantom{ssssssssssssssssssss}\qedhere$$
\end{proof}

    If $J$ is finite, \Cref{submoexist} reduces to the classical feasibility theorem for discrete submodular flows: 
\begin{coro}[Edmonds--Giles Feasibility Theorem \cite{Edmond1975}]\label{coro:edmonds-giles}
    Let $G=(V, E)$ be a finite directed graph with edge set $E \subseteq V \times V$. Let $l, u \colon E \to \R$ be edge capacity functions satisfying $l \le u$, and let $b \colon 2^V \to \R$ be a submodular set function with $b(\emptyset) = b(V) = 0$.
    We use $\delta^+(X)$ and $\delta^-(X)$ denote the set of edges leaving and entering $X$, respectively.
    Then the following two statements are equivalent.
    \begin{enumerate}
        \item There exists a discrete flow $x \colon E \to \R$ satisfying $$l(e) \le x(e) \le u(e), \;\forall \,e \in E;\quad \partial x(X)  \le b(X) ,  \;\forall \, X \subseteq V,$$
        where $$\partial x(X)=\sum_{e\in \delta^+(X)}x(e)- \sum_{e\in \delta^-(X)}x(e).$$
        \item For any $X \subseteq V$,
        $$\sum_{e\in \delta^+(X)}l(e) - \sum_{e\in \delta^-(X)}u(e) \le b(X).$$
    \end{enumerate}
\end{coro}
\begin{proof}
    We can naturally embed this discrete problem into our measurable framework. Let $J = V$ and $\B=2^V$. Define capacity measures $\alpha, \beta \in \mBB$ as follows:
    \[
        \alpha(S) := \sum_{e \in S \cap E} l(e), \qquad \beta(S) := \sum_{e \in S \cap E} u(e) ,\;   S \subseteq V \times V.
    \]
    Then $(J,\B,\alpha,\beta,b)$ is a submodular network system. Any feasible measure \(\mu\) corresponds bijectively to a discrete flow
\(x:E\to\mathbb R\) by setting \(x(e)=\mu(\{e\})\). Conversely, every such
discrete flow defines a measure supported on \(E\). Then the corollary follows by \Cref{submoexist}.
\end{proof}

\begin{exmp}\label{exmp:nonatomic}
We give a simple example showing that \Cref{submoexist} is not merely a reformulation of \Cref{coro:edmonds-giles}. Let $J=[0,1]$, $\B$ be the Borel $\sigma$-algebra generated by open intervals and $\lambda$ be the Lebesgue measure on $J$. Put
\[
    \Omega:=\{(x,y)\in[0,1]^2:x<y\},
    \qquad
    \psi:=(\lambda\otimes\lambda)|_\Omega.
\]
Thus $\psi$ may be viewed as a continuum of directed arcs from smaller points to larger points. Fix $0<r\le1$, and define lower and upper capacity measures by
\[
    \alpha:=r \psi,
    \qquad
    \beta:=\psi.
\]
Then $\alpha\le\beta$. For a parameter $q\ge0$, define
\[
    \varphi_q(X):=q\lambda(X)\lambda(X^c),\qquad X\in\B.
\]
The function $\varphi_q$ is bounded and continuous. Moreover, it is submodular, since for any $A,B\in\B$,
\[
\varphi_q(A)+\varphi_q(B)-\varphi_q(A\cup B)-\varphi_q(A\cap B)
=2q\lambda(A\setminus B)\lambda(B\setminus A)\ge0.
\]
Also $\varphi_q(\emptyset)=\varphi_q(J)=0$.

For every $X\in\B$, we have
\[
\alpha(X\times X^c)-\beta(X^c\times X)
\le \alpha(X\times X^c)
=r m(X\times X^c)
\le r(\lambda\otimes\lambda)(X\times X^c)
=r\lambda(X)\lambda(X^c).
\]
Hence, if $q\ge r$, then
\[
    \alpha(X\times X^c)-\beta(X^c\times X)\le\varphi_q(X)
\]
for all $X\in\B$. By \Cref{submoexist}, there exists a feasible submodular flow $\mu\in\mBB$ satisfying $\alpha\le\mu\le\beta$ and $\partial\mu\le\varphi_q$.

The threshold $q\ge r$ is sharp for this family. Indeed, if $q<r$ and $X=[0,t]$ with $0<t<1$, then
\[
    \alpha(X\times X^c)=rt(1-t),
    \qquad
    \beta(X^c\times X)=0,
\]
while $\varphi_q(X)=qt(1-t)$. Thus the cut condition in \Cref{submoexist} fails, and no feasible submodular flow exists when $q<r$.
\end{exmp}
\subsection{Optimality for submodular network system with cost}
In this subsection, we fix $\nu\in\mpBB$ and restrict $\mBB$ to those measures \emph{absolutely continuous} with respect to $\nu$, so that we can apply~\Cref{convexseparation}. We first give the following lemma about optimizing values of the integral of a potential on measures in the base polytope of a submodular set function.

\begin{lem}\label{maxminmuF}
    Given $(J,\B)$ and a bounded continuous submodular set function $\varphi$ on $\B$ with $\varphi(\emptyset)=\varphi(J)=0$, let $\lambda_\varphi\in \mpB$ be the measure given by~\Cref{lem:countablyadditive} such that every $\mu\in\bmm(\varphi)$ satisfies $\mu\ll\lambda_\varphi$.
    Then for any potential $F\colon J\times J\to\mathbb{R}$ with $F(x,y)=f(x)-f(y)$ with $f \ge 0$ on $J$, we have 
     \begin{align*}
        \max&\{\mu(F)\;:\; \mu\in\mBB,\; \partial\mu \leq \varphi\}=\hat{\varphi}(f)\\
        \min&\{\mu(F)\;:\; \mu\in\mBB,\; \partial\mu \leq \varphi\}=-\hat{\varphi}(-f).
    \end{align*}
    Moreover, the maximum and minimum can be attained by signed measures absolutely continuous with respect to $\lambda_\varphi\otimes \lambda_\varphi$. 
\end{lem}

\begin{proof}
    Assume $0\leq f< C$ and let $A_t=\{x\in J\;:\;f(x)\geq t\}$. By~\Cref{potentialint}, $$F(x,y)=\int_0^C\big(\mathbf{1}_{A_t}(x)-\mathbf{1}_{A_t}(y)\big)\,\dt.$$
    For any $\mu\in \mBB$, Fubini's Theorem gives
    \begin{align*}
        \mu(F)&=\int_{J\times J}\int_0^C\big(\mathbf{1}_{A_t}(x)-\mathbf{1}_{A_t}(y)\big)\,\dt\,\mathrm{d}\mu(x,y)=\int_0^C\big(\mu(A_t\times J)-\mu(J\times A_t)\big)\,\dt\\
        &=\int_0^C\partial \mu (A_t)\,\dt\ \leq \ \int_0^C\varphi(A_t)\,\dt \ =\ \hat{\varphi}(f).
    \end{align*}
    Similarly, with $B_t=\{x\in J\,:\,f(x)\leq t\}$ we have $$F(x,y)=-\int_{0}^C\big(\mathbf{1}_{B_t}(x)-\mathbf{1}_{B_t}(y)\big)\, \dt,$$ and using $\varphi(J)=0$,
    \begin{align*}
        \mu(F)&=-\int_0^C\big(\mu(B_t\times J)-\mu(J\times B_t)\big)\,\dt=-\int_0^C\partial \mu (B_t)\,\dt\geq -\int_0^C\varphi(B_t)\,\dt\\
        &=-\int_0^C\varphi(B_{C-t})\,\dt=-\int_0^C\varphi(\{C-f\geq t\})\,\dt=-\hat{\varphi}(C-f)=-\hat{\varphi}(-f).
    \end{align*}

    We next prove that the maximum and minimum can be attained.
    Let
\[
\mathcal C := \{A_t : 0\le t\le C\}.
\]
Then $\mathcal C$ is a lattice containing $\emptyset$ and $J$.
Moreover, since $\mathcal C$ is a chain, the restriction
$\varphi|_{\mathcal C}$ is modular.
Therefore, by \Cref{strongerseparation},
there exists $\alpha\in\bmm(\varphi)$ such that
\[
\alpha(S)=\varphi(S)\qquad \text{for all }S\in\mathcal C.
\]
    It is therefore enough to construct a flow \(\mu\in \M(\mathcal B^2)\)
such that \(\partial\mu=\alpha\). Indeed, for such a \(\mu\),
\[
\mu(F)
=
\int_0^C \partial\mu(A_t)\,\dt
=
\int_0^C \alpha(A_t)\,\dt
=
\int_0^C \phi(A_t)\,\dt
=
\widehat\phi(f).
\]
    Let $\alpha=\alpha_1-\alpha_2$ be the Hahn decomposition of $\alpha$ and $\alpha_i$ is supported on $X_i$, $i=1,2$, such that $J=X_1\sqcup X_2$.
    Since $\alpha(J)=0$, we define $a=\alpha_1(X_1)=\alpha_2(X_2)$.
    If $a=0$, then $\alpha=0=\partial 0$ and $\mu=0$ attains maximum. Now we assume $a>0$.
    Let $$\nu_\alpha=\dfrac{2}{a}(\alpha_1+\alpha_2)\otimes (\alpha_1+\alpha_2)\in \mpBB.$$ 
    
    We claim that for any $S\in \B$ we have $$\nu_\alpha(S\times S^c)\geq \alpha(S).$$ Let $S=S_1\sqcup S_2$ such that $S_1\subseteq X_1$ and $S_2\subseteq X_2$. Then $\alpha(S)=\alpha_1(S_1)-\alpha_2(S_2)$. Let $T_1=X_1\setminus S_1$ and $T_2=X_2\setminus S_2$.
        If $\alpha_1(S_1)\leq \alpha_1(T_1)$. Then $$\nu_\alpha(S\times S^c)\geq\dfrac{2}{a}\cdot\alpha_1(S_1)\cdot \alpha_1(T_1)\geq \dfrac{2}{a}\cdot\alpha_1(S_1)\cdot\dfrac{a}{2}\geq\alpha_1(S_1)\geq \alpha(S).$$
        If $\alpha_1(S_1)> \alpha_1(T_1)$. Since $\alpha_1(S_1)-\alpha_2(S_2)=\alpha_2(T_2)-\alpha_1(T_1)\leq \alpha_2(T_2)$, we have
        $$\nu_\alpha(S\times S^c)\geq\dfrac{2}{a}\cdot\alpha_1(S_1)\cdot \alpha_2(T_2)\geq \dfrac{2}{a}\cdot\dfrac{a}{2}\cdot\alpha_2(T_2)\geq\alpha(S).$$
    By~\Cref{flowexistence}, there exists $\mu\in \mBB$ such that $0\leq \mu\leq \nu_\alpha$ and $\partial \mu=\alpha$. Then $\mu(F)=\hat{\varphi}(f)$.
    Note that $\alpha_1,\alpha_2\ll \lambda_\varphi$ and thus $\nu_\alpha\ll\lambda_\varphi\otimes \lambda_\varphi$. As a result, $\mu\ll\lambda_\varphi\otimes \lambda_\varphi$.

    For the minimum, applying the same argument to the chain $\mathcal{C}':=\{B_t\}_{0\leq t\leq C}$ gives a signed measure $\beta\leq \varphi$ with $\beta|_{\mathcal{C}'}=\varphi|_{\mathcal{C}'}$ and a flow $\mu'\in\mBB$ with $\partial \mu'=\beta$ and $\mu'(F)=-\hat{\varphi}(-f)$. Still we have $\mu'\ll\lambda_\varphi\otimes \lambda_\varphi$.    
\end{proof}

The following theorem is a submodular flow version of \cite[Theorem~4.6]{lovasz2021flows}, with \Cref{maxminmuF} providing the potential-optimization step. For functions on $J$, we denote by $g^+$ and $g^-$ the two functions $ g\vee 0$ and $ (-g)\vee 0$ and it is easy to see that $g=g^+-g^-$.

\begin{thm}[Optimality Theorem]\label{thm:minicostsubmo}
     Given a submodular network system with cost $(J,\B,\alpha,\beta,\varphi,v)$, there exists a feasible submodular flow $\mu$ such that $\mu(v)=c$ if and only if there exists a feasible submodular flow and for any potential $F(x,y)=f(x)-f(y)$ and $f\geq 0$, we have
     \begin{align}
         \label{eq:dualineqonpotentialF}
         \alpha((F+v)^+)-\beta((F+v)^-)-\hat{\varphi}(f)\ \leq \ c  \ \leq \ \beta((F+v)^+)-\alpha((F+v)^-)+\hat{\varphi}(-f).
     \end{align}
     Consequently,
\[
\inf_{\mu \in \F_\varphi(J,\B,\alpha,\beta)}\mu(v)=
\sup_{F=f(x)-f(y),\ f\ge0}
\left[ \alpha((F+v)^+)-\beta((F+v)^-)-\hat{\varphi}(f)\right],
\]
and
\[
\sup_{\mu \in \F_\varphi(J,\B,\alpha,\beta)}\mu(v)=
\inf_{F=f(x)-f(y),\ f\ge0}
\left[\beta((F+v)^+)-\alpha((F+v)^-)+\hat{\varphi}(-f)\right],
\]
where         \[\F_\varphi(J,\cB,\alpha,\beta):=
        \{\mu\in \mBB:\alpha\le\mu\le\beta,
        \ \partial\mu\le\varphi\}.\]
\end{thm}

\begin{proof}
    \emph{Necessity.} If such a flow $\mu$ exists, then for any potential $F$ and any measurable $g$ we have the elementary bounds
    \[
      \alpha(g^+)-\beta(g^-)\ \le\ \mu(g)\ \le\ \beta(g^+)-\alpha(g^-),
    \]
    valid whenever $\alpha\le \mu\le \beta$. With $g=F+v$ and using \Cref{maxminmuF} we have $$\beta((F+v)^+)-\alpha((F+v)^-)\ \geq \ \mu(F+v) \ =\ \mu(F)+\mu(v) \ \geq \ c -\hat{\varphi}(-f),$$  giving the right-hand inequality in~\eqref{eq:dualineqonpotentialF}.
    Similarly, $$\alpha((F+v)^+)-\beta((F+v)^-)\ \leq \ \mu(F+v)\ =\ \mu(F)+\mu(v)\ \leq \ c+\hat{\varphi}(f),$$ which yields the left-hand inequality.

    \emph{Sufficiency.}
    Let $\cS$ be the set of all submodular flows, $\cH=\{\mu \in \mBB: \mu(v)=c\}$ and $\cX=\{\mu \in \mBB: \alpha\le \mu\le \beta\}$.
    For $\epsilon>0$, let $\cS_\epsilon=B_\epsilon(\cS)$, $\cH_\epsilon=B_\epsilon(\cH)$ and $\cX_\epsilon=B_\epsilon(\cX)$ to be the $\epsilon$-ball in the total variation norm.
    We claim that $$\cS_\epsilon\cap \cH_\epsilon\cap \cX_\epsilon\neq \emptyset$$ for any $\epsilon>0$. 
    
     Suppose, for contradiction, that $\cS_\epsilon\cap \cH_\epsilon\cap \cX_\epsilon=\emptyset$ for some $\epsilon>0$. By~\Cref{thm:separatinghyperplane} there exist bounded linear functionals $\cL_i$ on $\mBB$ and $a_i\in \R$ ($i=1,2,3$) such that $$\cL_1+\cL_2+\cL_3=0,\;a_1+a_2+a_3=0,\;\cL_i(\mu_i)\geq a_i,\;\sum_{i=1}^3\cL_i(\mu_i)>0,\;\mu_1\in \cS,\; \mu_2\in\cH,\; \mu_3\in \cX.$$
     Note that for any submodular flow $\mu$ and circulation $\nu$, we have $\mu+a\nu$ is a submodular flow for any $a\in\R$. Thus $\cL_1(\mu)+a\cL_1(\nu)=\cL_1(\mu+a\nu)\geq a_1$ for any $a\in \R$ and so $\cL_1(\nu)=0$ for any circulation $\nu$. 
     
     By~\Cref{lem:countablyadditive} there exists some $\lambda_\varphi\in \mpB$ such that for any $\mu\in \bmm(\varphi)$, $\mu\ll \lambda_\varphi$. Let $\nu_0=\lambda_\varphi\otimes \lambda_\varphi+|\alpha|+|\beta|$.
     By~\Cref{convexseparation} there exists a potential $F(x,y)=f(x)-f(y)$ with $0\leq f< C$ such that $\cL_1(\mu)=\mu(F)$ for all $\mu \in \mBB$ with $\mu\ll \nu_0$.
     By~\Cref{maxminmuF} there exists $\nu_1,\nu_2\in\mpBB$ such that $\nu_1,\nu_2\ll \nu_0$, $\nu_1(F)= -\hat{\varphi}(-f)$ and $\nu_2(-F)=\hat{\varphi}(-f)$.
    Similarly, since $\H$ is a hyperplane and $\cL_2(\H)\geq a_2$,
    there exists $b\in \R$ such that for any $\mu\in \mBB$ we have $\cL_2(\mu)=b\mu(v)$. Then $\mathcal{L}_3(\mu)=-\mu(F)-b\mu (v)=-\mu(F+bv)$.
    By the above discussion, for any $\mu_1\in \cS,\; \mu_2\in\cH,\; \mu_3\in \cX$ with $\mu_1\ll\nu_0$ we have $$0=a_1+a_2+a_3<\mu_1(F)+bc-\mu_3(F+bv),$$
    and thus 
    $$\beta((F+bv)^+)-\alpha((F+bv)^-)<\mu_1(F)+bc.$$
If \(b\ne0\), divide the separating inequality by \(|b|\) and replace
\(F\) by \(F/|b|\). Thus it suffices to consider \(b=1\) and \(b=-1\).
The case \(b=0\) is treated separately.

    If $b=1$, let $\mu_1=\nu_1$ we have
    $$\beta((F+v)^+)-\alpha((F+v)^-)<c-\hat{\varphi}(-f),$$ contradicting the right-hand side of~\eqref{eq:dualineqonpotentialF}.

    If $b=-1$, let $\mu_1=\nu_2$ we have $$\alpha((-F+v)^+)-\beta((-F+v)^-)>c+\hat{\varphi}(-f),$$ contradicting the left-hand side of~\eqref{eq:dualineqonpotentialF}. 

    If $b=0$, let $\mu_1=\nu_1$ we have $$\beta(F^+)-\alpha(F^-)<-\hat{\varphi}(-f).$$
    Let $B_t=\{C-f\geq t\}$. Then $\hat{\varphi}(-f)=\hat{\varphi}(C-f)=\int_0^C\varphi(B_t)\;\dt$. By~\Cref{submoexist} and Fubini's Theorem,
    \begin{align*}
        -\hat{\varphi}(-f)&=-\int_0^C\varphi(B_t)\;\dt\leq \int_0^C \big(\beta(B_t^c\times B_t)-\alpha(B_t\times B_t^c)\big)\;\dt\\
        &=\int_0^C \int_{B_t^c\times B_t} 1 \;\text{d}\beta(x,y) - \int_{B_t\times B_t^c} 1 \;\text{d}\alpha(x,y)    \; \dt\\
        &=\int_{\{(x,y)\;:\; f(x)>f(y)\}} \int_{f(y)}^{f(x)} 1 \; \dt\;\text{d}\beta(x,y)-\int_{\{(x,y)\;:\; f(x)<f(y)\}} \int_{f(x)}^{f(y)} 1 \; \dt\;\text{d}\alpha(x,y)\\
        &=\beta(F^+)-\alpha(F^-),
    \end{align*}
    a contradiction.

    Therefore $\cS_\epsilon\cap \cH_\epsilon\cap \cX_\epsilon\neq \emptyset$ for every $\epsilon>0$.

    For each $n\ge 1$, choose $$\xi_n\in \cS_{2^{-n}}\cap \cH_{2^{-n}}\cap \cX_{2^{-n}}.$$
    By the definition of the $2^{-n}$-neighborhoods, we may choose $\mu_n\in \cS$, $\eta_n\in \cH$, and $\omega_n\in \cX$ such that $$\|\xi_n-\mu_n\|<2^{-n}, \|\xi_n-\eta_n\|<2^{-n}, \|\xi_n-\omega_n\|<2^{-n}.$$
    Hence $\|\mu_n-\omega_n\|<2^{1-n}$ and $\|\eta_n-\omega_n\|<2^{1-n}$.

    Since $\omega_n\in\cX$, we have $\alpha\le \omega_n\le \beta$ for every $n$. Therefore $|\omega_n|\le 2|\alpha|+|\beta|$ for every $n$.
    By~\Cref{compactmeasure}, after passing to a subsequence, there exists $\omega\in\mBB$ such that $\omega_n(U)\to \omega(U)$ for every $U\in\B^2$.

    We claim that the same subsequence satisfies $\mu_n(U)\to \omega(U)$ and $\eta_n(U)\to \omega(U)$ for every $U\in\B^2$.
    Indeed, for every $U\in\B^2$,
    \[
    |\mu_n(U)-\omega(U)|
    \le
    |\mu_n(U)-\omega_n(U)|+|\omega_n(U)-\omega(U)|
    \le
    \|\mu_n-\omega_n\|+|\omega_n(U)-\omega(U)|\to 0,
    \]
    and similarly $\eta_n(U)\to\omega(U)$.

    We now verify that $\omega\in \cS\cap \cH\cap \cX$. Since $\omega_n\in\cX$, for every $U\in\B^2$ we have $\alpha(U)\le \omega_n(U)\le \beta(U)$. Passing to the limit gives $\alpha(U)\le \omega(U)\le \beta(U)$ for all $U\in\B^2$, so $\omega\in\cX$. Since each $\eta_n\in\cH$, we have $\eta_n(v)=c$. Since $v$ is bounded and measurable and $\eta_n\to\omega$ setwise, it follows that $\omega(v)=\lim_{n\to\infty}\eta_n(v)=c$. Hence $\omega\in\cH$.
    Since each $\mu_n\in\cS$, for every $A\in\B$, we have $$\partial\mu_n(A)=\mu_n(A\times A^c)-\mu_n(A^c\times A)\le \varphi(A).$$ Passing to the limit, using the setwise convergence of $\mu_n$ to $\omega$ on the sets
    $A\times A^c$ and $A^c\times A$, we obtain $\partial\omega(A)\le \varphi(A)$ for all $A\in\B$.
    Thus $\omega\in\cS$.

    Therefore $\omega\in \cS\cap \cH\cap \cX$, so $\omega$ is a feasible submodular flow with cost $c$.
\end{proof}

\section{Extreme flows and measurable residual graphs}\label{sec4}

\subsection{Extreme submodular flows}\label{subsec:extremesubmodularflow}
Let $\Lambda_0:=|\alpha|+|\beta|$.
Recall that \[\F_\varphi(J,\cB,\alpha,\beta):=
        \{\mu\in \mBB:\alpha\le\mu\le\beta,
        \ \partial\mu\le\varphi\}.\]
If \(\mu\in\mathcal F_\varphi(J,\mathcal B,\alpha,\beta)\), then
\(\alpha\le\mu\le\beta\), and hence $|\mu|\le \Lambda_0$.
Thus every feasible flow is absolutely continuous with respect to
\(\Lambda_0\). Identifying such a measure \(\mu\) with its
Radon--Nikodým derivative
$f_\mu:=\frac{\mathrm d\mu}{\mathrm d\Lambda_0}$,
the set
\[
    K_{\Lambda_0}:=
    \{\mu\in M(\mathcal B^2): |\mu|\le\Lambda_0\}
\]
is identified with the closed unit ball of \(L^\infty(\Lambda_0)\).

We equip \(K_{\Lambda_0}\) with the weak-* topology $\sigma(L^\infty(\Lambda_0),L^1(\Lambda_0))$,
that is, the topology in which \(f_i\to f\) if and only if
\[
    \int f_i h\,d\Lambda_0\to \int f h\,d\Lambda_0
    \qquad\text{for every }h\in L^1(\Lambda_0).
\]
By the Banach--Alaoglu theorem, \(K_{\Lambda_0}\) is compact in this topology.
On \(K_{\Lambda_0}\), this topology agrees with the topology of setwise
convergence of measures. Indeed, indicators of measurable sets belong to
\(L^1(\Lambda_0)\), and conversely simple functions are dense in
\(L^1(\Lambda_0)\), while all densities in \(K_{\Lambda_0}\) are uniformly
bounded by \(1\).

The capacity constraints are weak-* closed, since they are equivalent to
\[
    \alpha(U)\le\mu(U)\le\beta(U),
    \qquad U\in\mathcal B^2,
\]
and each map \(\mu\mapsto\mu(U)\) is weak-* continuous. Similarly, for every
\(X\in\mathcal B\), the boundary functional
\[
    \mu\mapsto\partial\mu(X)
    =
    \mu(X\times X^c)-\mu(X^c\times X)
\]
is weak-* continuous. Hence the constraints
\[
    \partial\mu(X)\le\varphi(X),\qquad X\in\mathcal B,
\]
define a weak-* closed subset of \(K_{\Lambda_0}\).

Consequently, \(\mathcal F_\varphi(J,\mathcal B,\alpha,\beta)\) is a weak-*
compact convex set. By the Krein--Milman theorem, it is the weak-* closed
convex hull of its extreme points.

We now characterize the extreme points.
Recall that a point $x$ in a convex set $C$ is an \emphd{extreme point} if there is no two distinct points $y_1,y_2\in C$ such that $x=\frac{1}{2}( y_1+y_2)$.

We begin with some basic calculation rules.
Let $\Lambda=\Lambda_0+\Lambda_0^*$, then any flow $\mu$ with $\alpha\le \mu\le \beta$ is absolutely continuous with respect to $\Lambda$. We denote by $f_\gamma\colon J\times J\to \R$ the \emphd{Radon--Nikodým derivative} for any signed measure $\gamma\ll \Lambda$.
The function $f_\gamma^*\colon J\times J\to \R$, $f_\gamma^*(x,y)=f_\gamma(y,x)$ is called the \emphd{transpose} of $f_\gamma$. Note that $f^*_\gamma=f_{\gamma^*}$. For any two functions $f,g\colon J\times J\to \R$, let $f\vee g=\max\{f,g\}$, $f\wedge g=\min\{f,g\}$, $f_+=f\vee 0$, $f_-=f\wedge 0$ and $f|_E=f\mathbf{1}_E$ for any $E\in\B^2$. Here $\mathbf{1}_E$ is the indicator function of subset $E$.
If the integral makes sense, we define the \emphd{measure induced by} $f$ to be $$f\Lambda \colon \B^2\to\R,\; f\Lambda(X)=\int_X f(x,y)\,\mathrm{d}\Lambda.$$
Functions are identified when they agree $\Lambda$-almost everywhere.

The following proposition is immediate from the Radon--Nikod\'ym representation with respect to the common dominating measure $\Lambda$. 
\begin{prop}\label{lem:basic}
    For any two functions $f,g\colon J\times J\to \R$, flows $\gamma,\gamma_1,\gamma_2$ and measurable $E,E_1,E_2$, the following calculation rules hold:
    \begin{enumerate}[label=(\arabic*)]
        \item\label{property1} $\partial \gamma ^{*}=-\partial \gamma$;
        \item\label{property2} if $\gamma,\gamma_1,\gamma_2\ll \Lambda$, then $f_\gamma\Lambda=\gamma$, $(f_{\gamma_1}\vee f_{\gamma_2} )\Lambda=\gamma_1\vee\gamma_2$ and $(f_{\gamma_1}\wedge f_{\gamma_2}) \Lambda=\gamma_1\wedge\gamma_2$;
        \item\label{property3} $(f\wedge g)^*=f^*\wedge g^*$ and $(\gamma_{1}\wedge \gamma_{2})^{*}=\gamma_{1}^{*}\wedge \gamma_{2}^{*}$;
        \item\label{property4} $(f-g)_+=f-f\wedge g$ and $(\gamma_{1}-\gamma_{2})_{+}=\gamma_{1}-\gamma_{1}\wedge \gamma_{2}$;
        \item\label{property5} $(f_+)^*=(f^*)_+$ and $(\gamma _{+})^{*}=(\gamma ^{*})_{+}$;
        \item\label{property6} $f_\gamma=f_\gamma|_E$ if $f|_{E^c}=0$, and $\gamma=\gamma|_{E}$ if $\gamma$ is supported on $E$;
        \item\label{property7} If $J \times J =E_{1}\sqcup E_{2}$, then $$f\wedge g=f|_{E_1}\wedge g|_{E_1}+f|_{E_2}\wedge g|_{E_2}\quad \textrm{ and } \quad\gamma_{1}\wedge \gamma_{2}=\gamma_{1}|_{E_{1}}\wedge \gamma_{2}|_{E_{1}}+\gamma_{1}|_{E_{2}}\wedge \gamma_{2}|_{E_{2}}.$$
    \end{enumerate}
\end{prop}

To start, we consider particular feasible flows with given marginals.
Given two measures $\sigma,\tau\in\mpB$ with $\sigma(J)=\tau(J)$, the set of all feasible $\sigma$-$\tau$ flows is denoted by \[
\F=\F(J,\B,\alpha,\beta,\sigma,\tau)
 \ :=\ \{\mu\in\mBB:\ \alpha\le \mu\le \beta,\ \mu^1-\mu^2=\sigma-\tau\}.
\]
Similar to $\F_\varphi$, the set $\F$ is a weak-* closed convex subset of $\mBB$.
A flow $\mu\in \F$ is an extreme point of $\F$ if and only if there is no non-zero flow $\gamma\in \mBB$ such that $\mu\pm \gamma\in\F$.
Note that the condition $\mu \pm \gamma \in \F$ requires $\gamma$ to be a circulation.
The capacity constraints $\alpha \le \mu \pm \gamma \le \beta$ are equivalent to $$\pm\gamma \le \mu - \alpha\quad\text{ and }\quad \pm\gamma \le \beta - \mu.$$
Thus we define
\[ \eta=\eta(\mu;\alpha,\beta):=(\beta-\mu)\wedge(\mu-\alpha). \]
Then $\mu$ is an extreme point of $\F$ if and only if there is no non-zero circulation $\gamma$ such that $-\eta\leq \gamma\leq \eta.$

As in the finite case, we can transform this condition concerning signed measures to concern only measures with a generalized notion of \emph{residual graphs} as below.

A nonnegative flow $\mu$ is called \emphd{trivial} if $\mu=\mu\wedge\mu^*$ and is called \emphd{reduced} if $\mu\wedge\mu^*=0$.

By Hahn decomposition we have the following lemma.

\begin{lem}\label{propersupport}
    If $\mu$ is a reduced flow, there exists $E\in \B^2$ such that $E\cap E^*=\emptyset$ and $\mu$ is supported on $E$.
\end{lem}
\begin{proof}
    Since $\mu\wedge\mu^*=0$, there exists a Borel partition $J\times J=E_1\sqcup E_2$ such that $\mu(E_2)=\mu^*(E_1)=0$. Then $$\mu(E_1\setminus E_1^*)=\mu(J\times J)-\mu(E_1^*)=\mu(J\times J)-\mu^*(E_1)=\mu(J\times J),$$
    so $\mu$ is supported on $E:= E_1\setminus E_1^*$, and clearly $E\cap E^*=\emptyset$. 
\end{proof}

Given a flow $\mu \in \F$, the \emphd{measurable residual graph} of $\mu$ is defined as 
\begin{equation}
R=R(\mu;\alpha,\beta):= \eta(\mu;\alpha,\beta)+\eta(\mu;\alpha,\beta)^*.
\end{equation}
The \emphd{trivial circulation graph} of $\mu$ is defined as $$T=T(\mu;\alpha,\beta):=\eta(\mu;\alpha,\beta)\wedge \eta(\mu;\alpha,\beta)^*.$$

As the classical residual graph, the measure $R$ describes how much a given flow can be adjusted in either direction—increasing ($\eta$) or decreasing ($\eta^*$); while $T$ captures flows whose weight coincide on every pair of arcs that has opposite direction, which vanishes for oriented graphs.
If $J$ is finite, $f_R$ is the weight function of the residual graph, and a reduced circulation in $R$ corresponds to some cycles with positive weight in the residual graph. Further, $T$ corresponds to cycles of length 2 with positive weight in the residual graph.

\begin{lem}\label{lem:extremeflow}
    A flow $\mu\in\F$ is an extreme point of $\F$ if and only if $T=0$ and there is no reduced circulation $\gamma$ such that $0< \gamma\leq R$.
\end{lem}
\begin{proof}
First assume $\mu$ is an extreme point of $\F$.
    Since $T=T^*$, \Cref{lem:basic} \ref{property1} gives $2\partial T=\partial T+\partial T^*=0$, so $T$ is a circulation. Assume $T\neq 0$.
    Recall that $\eta = (\beta-\mu)\wedge(\mu-\alpha) \ge 0$ since $\alpha \le \mu \le \beta$. By definition, $T = \eta \wedge \eta^*$, which implies $0 \le T \le \eta$, and thus $-\eta \le \pm T \le \eta$.
    Because $\eta \le \beta - \mu$ and $\eta \le \mu - \alpha$, the bounds $-\eta \le \pm T \le \eta$ guarantee the capacity constraints $\alpha \le \mu \pm T \le \beta$. Furthermore, $T$ is a circulation, contradicting that $\mu$ is extreme. Thus $T = \eta \wedge \eta^* = 0$ and $\eta$ is a reduced flow.
    
    Suppose there is a nonzero reduced circulation $\gamma$ such that $0< \gamma\leq R$. Let $$f_1:=f_\gamma\wedge f_\eta, \qquad f_2:=f_{\gamma}\wedge f_{\eta^*}.$$ 
    Then $f_1\Lambda\leq \eta$ and $f_2\Lambda \leq \eta^*$.
    Since $\eta$ is reduced, 
    $$f_\gamma=f_\gamma\wedge f_R=(f_\gamma\wedge f_R)|_{\{f_\eta>0\}}+(f_\gamma\wedge f_R)|_{\{f_{\eta^*}>0\}}=f_1+f_2.$$
    Thus $f_\gamma\Lambda=f_1\Lambda+f_2\Lambda$.
    Define $$\mu_1:=\mu+f_1\Lambda-f_2^*\Lambda,\qquad \mu_2:=\mu-f_1\Lambda+f_2^*\Lambda.$$
    Then $$\partial\mu_1=\partial\mu_2=\partial \mu+\partial (f_1\Lambda)+\partial (f_2\Lambda)=\partial\mu +\partial \gamma=\partial\mu,$$ and similarly $\partial\mu_2=\partial\mu$. Thus $\mu_1,\mu_2\in\F$ and $\mu=(\mu_1+
    \mu_2)/2$. Extremality gives $\mu_1=\mu_2$, hence $f_1=f_2^*$, which implies
    $$\gamma\wedge\gamma^*=(f_1\Lambda+f_2\Lambda)\wedge(f_1^*\Lambda+f_2^*\Lambda)=\gamma,$$ contradicting that $\gamma$ is reduced.

    Now we assume $T=0$ and there is no nonzero reduced circulation $\gamma$ such that $0\leq \gamma\leq R$. Suppose, for contradiction, that $\mu$ is not an extreme point of $\F$. Then there exist distinct flows $\mu_1,\mu_2\in\F$ such that $\mu=\dfrac{\mu_1+\mu_2}{2}$, which implies $\mu_1-\mu = \mu-\mu_2$. 
    
    Let $\delta_1:=(\mu_1-\mu)_+$ and $\delta_2:=(\mu-\mu_1)_+$ be the positive and negative variations of the signed measure $\mu_1-\mu$, meaning $\mu_1-\mu = \delta_1-\delta_2$.
    Because $\mu_1 \le \beta$ and $\mu_2 \ge \alpha$, we can bound $\delta_1$ simultaneously in two ways:
    \[ \delta_1 = (\mu_1-\mu)_+ \le (\beta-\mu)_+ = \beta-\mu \quad \text{and} \quad \delta_1 = (\mu-\mu_2)_+ \le (\mu-\alpha)_+ = \mu-\alpha. \]
    Therefore, $\delta_1 \le (\beta-\mu)\wedge(\mu-\alpha) = \eta$. Taking the transpose yields $\delta_1^* \le \eta^*$. 
    Thus, $\delta_1 \wedge \delta_1^* \le \eta \wedge \eta^* = T$. Since $T=0$ by our assumption, we have $\delta_1 \wedge \delta_1^* = 0$, meaning $\delta_1$ is a reduced flow. By symmetry, $\delta_2\leq \eta$ and $\delta_2$ is also a reduced flow.

    By the Hahn decomposition theorem, $\delta_1$ and $\delta_2$ are mutually singular and there exist disjoint Borel sets $E_1, E_2 \subseteq J \times J$ such that $\delta_i$ is supported on $E_i$ for $i=1,2$. Since $\delta_1$ and $\delta_2$ are reduced, by \Cref{propersupport} we can additionally require $E_1 \cap E_1^* = \emptyset$ and $E_2 \cap E_2^* = \emptyset$. 

    We now construct a new flow $\delta := \delta_1 + \delta_2^*$. 
    First, since $\mu_1, \mu \in \F$ share identical marginals, we have $\partial(\mu_1 - \mu) = \partial\delta_1 - \partial\delta_2 = 0$. By \Cref{lem:basic} \ref{property1}, $\partial\delta = \partial\delta_1 + \partial\delta_2^* = \partial\delta_1 - \partial\delta_2 = 0$, meaning $\delta$ is a circulation.
    Second, since $\delta_1 \le \eta$ and $\delta_2 \le \eta$ (which implies $\delta_2^* \le \eta^*$), we inherently have $0 \le \delta \le \eta + \eta^* = R$.
    Finally, we verify that $\delta$ is reduced. The flow $\delta$ is supported on $E_1 \cup E_2^*$, and its transpose $\delta^*$ is supported on $E_1^* \cup E_2$. By our assumption, $$E_1 \cap E_1^* = E_2 \cap E_2^* = E_1 \cap E_2 = (E_2^* \cap E_1^*) = (E_1 \cap E_2)^* = \emptyset,$$ and thus the intersection of their supports
    $(E_1 \cup E_2^*) \cap (E_1^* \cup E_2) = \emptyset.$

    Therefore, $\delta \wedge \delta^* = 0$, meaning $\delta$ is a nonzero reduced circulation satisfying $0 \le \delta \le R$, a contradiction.
    \end{proof}

Now we are ready to give a characterization of extreme submodular flows. By considering common basic minorizing measures, we can transform the submodular flow problem to a flow problem with fixed marginals.

\begin{thm}\label{thm:extremesubmodularflow}
    For any $\mu\in \F_\varphi(J,\B,\alpha,\beta)$, let $\eta=\eta(\mu;\alpha,\beta)$, $R=R(\mu;\alpha,\beta)$ and $T=T(\mu;\alpha,\beta)$. Then    
    $\mu$ is an extreme point of $\F_\varphi(J,\B,\alpha,\beta)$ if and only if the following two statements hold:
    \begin{enumerate}
        \item\label{cond1} $\bmm(\kappa_{-\eta,\eta})\cap\bmm(\varphi-\partial\mu)\cap \bmm((\varphi-\partial\mu)^c)=\{0\}$;
        \item\label{cond2} $T=0$, and there is no reduced circulation $\gamma$ such that $0< \gamma\leq R$.
    \end{enumerate}
Here $(\varphi-\partial\mu)^c(X) := (\varphi-\partial\mu)(X^c)$ for any $X\in\B$, a submodular set function whose value for $\emptyset$ is 0.    
\end{thm}

\begin{proof}
    For any $\mu\in \F_\varphi(J,\B,\alpha,\beta)$, $\mu$ is an extreme point if and only if there is no $\gamma\in\mBB$ such that $\alpha\leq \mu\pm\gamma\leq \beta$ and $\partial\mu\pm\partial\gamma\leq \varphi$. The first condition is equivalent to say that $-\eta\leq \gamma\leq \eta,$ and the second condition is equivalent to say that $-(\varphi-\partial\mu) \leq \partial\gamma\leq \varphi-\partial\mu.$
    Let \[
        \psi:=\varphi-\partial\mu,
        \qquad
        \psi^c(X):=\psi(X^c).
\]
    Then both $\psi$ and $\psi^c$ are bounded continuous submodular set functions.
    
    For any $X\in\B$, since $\partial\gamma(X)+\partial\gamma(X^c)=0$ and
    $$\partial\gamma(X)+\psi(X)=\psi^c(X^c)-\partial\gamma(X^c). $$
    Thus the second condition is equivalent to the fact that $\partial\gamma\in \bmm(\psi)\cap \bmm(\psi^c)$.

    Assume $\mu$ is an extreme point.
    By~\Cref{partialinbmm}, we have $$\bmm(\kappa_{-\eta,\eta})=\{\partial \gamma \;:\;\gamma\in\mBB,\; -\eta\leq \gamma\leq \eta\}.$$
    Then $\bmm(\kappa_{-\eta,\eta})\cap\bmm(\psi)\cap \bmm(\psi^c)=\{0\}$, or otherwise there exists a nonzero flow $\gamma$ such that $-\eta\leq \gamma\leq \eta$ and $\partial\gamma\in \bmm(\psi)\cap \bmm(\psi^c)$, a contradiction. Further, since $$\F:=\F(J,\B,\alpha,\beta,(\partial\mu)_+,(\partial\mu)_-)\subseteq \F_\varphi,$$ $\mu$ is also an extreme point of $\F$. By~\Cref{lem:extremeflow}, condition~\ref{cond2} in the theorem holds.

    Conversely, assume the two conditions in the theorem hold and $\mu$ is not an extreme point. By the previous discussion, there exists a nonzero flow $\gamma$ such that $-\eta\leq \gamma\leq \eta$ and $-(\varphi-\partial\mu) \leq \partial\gamma\leq \varphi-\partial\mu$. Since $\bmm(\kappa_{-\eta,\eta})\cap\bmm(\psi)\cap \bmm(\psi^c)=\{0\}$, we have $\partial\gamma=0$ and thus $\mu\pm\gamma\in\F$. However, by~\Cref{lem:extremeflow} and condition~\ref{cond2}, $\mu$ is an extreme point of $\F$ and thus $\gamma=0$, a contradiction.
\end{proof}

\subsection{Extreme reduced flows}

In classical network flow theory, if the lower bound of edge capacity is 0, it is standard to consider \emphd{net flows} in order to eliminate transport redundancies. On a directed graph, simultaneous flows in opposite directions between two adjacent nodes (i.e., trivial $2$-cycles) do not contribute to the flow while increasing the cost. To remove this redundancy, one typically subtracts the weight on opposite edges simultaneously until one direction become zero.
We extend this notion to the measurable setting as follows. 
Fix a network system $(J,\B,0,\beta)$.

\begin{defi}
    The \emphd{reduction map} is defined as
    \[
    P\colon \mpBB\to \mpBB,
    \qquad
    P(\mu):=\mu-\mu\wedge \mu^*=(\mu-\mu^*)_+.
    \]
    We define $\mu\sim\nu$ if $P(\mu)=P(\nu)$, and then $\sim$ is an equivalence relation on $\mpBB$.
    Notice that $$P(\mu) - P(\mu)^* = (\mu-\mu^*)_+ - (\mu-\mu^*)_- = \mu - \mu^*.$$ Since $P(\mu)$ and $P(\mu)^*$ are mutually singular, $P(\mu) = P(\nu)$ if and only if $\mu - \mu^* = \nu - \nu^*$. Thus the equivalence relation $\sim$ is precisely the equivalence modulo the kernel of the linear operator $$L(\mu) := \mu - \mu^*.$$
    The equivalence class containing $\mu$ is denoted by $[\mu]$. Note that $[P(\mu)]=[\mu]$.
    
    The flow $P(\mu)$ is called the \emphd{reduced flow} of $\mu$, serving as the canonical non-negative representative of its equivalence class. For fixed $\sigma,\tau\in\mpB$ with $\sigma(J)=\tau(J)$, we denote all reduced feasible $\sigma$-$\tau$ flows by the quotient space:
    \[
    \widetilde{\F}= \widetilde{\F}(J,\B,\beta,\sigma,\tau):=\F(J,\B,0,\beta,\sigma,\tau)/\sim.
    \]
    
    Because $\sim$ is induced by the linear map $L$, the quotient space $\widetilde{\F}$ is affinely isomorphic to the image $L(\F)$, a convex subset of signed measures.
    Through this isomorphism, $\widetilde{\F}$ naturally inherits a rigorous convex structure, where convex combinations are defined by $$c[\mu_1] + (1-c)[\mu_2] := [c\mu_1 + (1-c)\mu_2],\qquad 0 \le c \le 1,$$ independent of the choice of representatives.
    We equip the quotient space $\widetilde{\F}$ with the metric induced by the total variation norm.
    That is, $$d( [\mu],[\nu] )= \inf\{  \| \mu_1-\nu_1 \|  \;:\;  [\mu_1]= [\mu],\; [\nu_1]= [\nu] \} .$$
    Then $\widetilde{\F}$ is a Banach space and it is easy to see that $d( [\mu],[\nu] )=\|P(\mu)-P(\nu) \| $.
\end{defi}

We give a characterization of extreme flows in $\widetilde{\F}$ as follows.
Given $[\mu]\in \widetilde{\F}$, define the residual graph as
$$R=R([\mu];\beta):= (\beta - P(\mu) + P(\mu)^*) \wedge (\beta^* - P(\mu)^* + P(\mu)).$$

\begin{lem}\label{lem:extremereducedflow}
For any $[\mu]\in \widetilde{\F}(J,\B,\beta,\sigma,\tau)$, $[\mu]$ is an extreme point of 
$\widetilde{\F}(J,\B,\beta,\sigma,\tau)$ if and only if there is no reduced circulation 
$\gamma$ such that $0<\gamma\le R([\mu];\beta).$
\end{lem}

\begin{proof}
Let $\F:=\F(J,\B,0,\beta,\sigma,\tau)$. Consider the linear map 
$L:\mBB\to\mBB$ defined by $L(\theta):=\theta-\theta^*$. By the definition of the equivalence relation, 
$[\theta_1]=[\theta_2]$ if and only if $L(\theta_1)=L(\theta_2)$. Hence $\widetilde{\F}$ is affinely 
isomorphic to $L(\F)$. Let
\[
    \G:=\{\nu\in\mBB:\nu^*=-\nu,\ \partial\nu=2(\sigma-\tau),\ \nu_+\le\beta\}.
\]
We claim that $L(\F)=\G$. Indeed, if $\theta\in\F$, then $L(\theta)^*=-L(\theta)$, 
$\partial L(\theta)=\partial(\theta-\theta^*)=2\partial\theta=2(\sigma-\tau)$, and 
$L(\theta)_+=(\theta-\theta^*)_+=P(\theta)\le\theta\le\beta$. Hence $L(\theta)\in\G$. Conversely, 
if $\nu\in\G$, then $\nu^*=-\nu$ and $v_-=(v_+)^*$. Thus $\nu=\nu_+-(\nu_+)^*$ and 
$\partial\nu=2\partial\nu_+$. Since $\partial\nu=2(\sigma-\tau)$, we get 
$\partial\nu_+=\sigma-\tau$. Together with $0\le\nu_+\le\beta$, this gives $\nu_+\in\F$, and 
$L(\nu_+)=\nu$. Therefore $L(\F)=\G$.

It follows that $[\mu]$ is extreme in $\widetilde{\F}$ if and only if $\nu:=L(\mu)$ is extreme in $\G$. 
Replacing $\mu$ by its reduced representative $P(\mu)$, we may assume that $\mu=P(\mu)$. Then 
$\nu=\mu-\mu^*$, $\nu_+=\mu$, and
\[
    R=R([\mu];\beta)
    =
    (\beta-\mu+\mu^*)\wedge(\beta^*-\mu^*+\mu)
    =
    (\beta-\nu)\wedge(\beta^*+\nu).
\]
Note that $\nu$ is not extreme in $\G$ if and only if there 
exists a nonzero signed measure $\gamma$ such that $\nu\pm\gamma\in\G$. This is equivalent to 
$$\gamma^*=-\gamma, \quad\partial\gamma=0,\quad (\nu\pm\gamma)_+\le\beta.$$ Since 
$\beta\ge0$, the inequalities $(\nu\pm\gamma)_+\le\beta$ are equivalent to $\nu\pm\gamma\le\beta$,  
and hence to
$\gamma\le\beta-\nu$ and $-\gamma\le\beta-\nu$. Taking the transpose of the second 
inequality gives $\gamma\le\beta^*+\nu$. Thus $[\mu]$ is extreme in $\widetilde{\F}$ if and only if there 
exists no nonzero signed measure $\gamma$ such that $\gamma^*=-\gamma$, $\partial\gamma=0$ and $\gamma\le(\beta-\nu)\wedge(\beta^*+\nu)=R$.

Assume $[\mu]$ is not extreme, pick such $\gamma$ and set $\rho:=\gamma_+$. Since $\gamma^*=-\gamma$, we have $\gamma=\rho-\rho^*$ and $\rho$ is reduced. 
Moreover, $\partial\gamma=\partial(\rho-\rho^*)=2\partial\rho$, so $\rho$ is a reduced circulation. Since both $\beta-\nu$ and $\beta^*+\nu$ are 
nonnegative measures, the inequalities $\gamma\le\beta-\nu$ and $\gamma\le\beta^*+\nu$ imply 
$\rho=\gamma_+\le(\beta-\nu)\wedge(\beta^*+\nu)=R$. Also $\gamma\ne0$ implies 
$\rho\ne0$, a contradiction.

Conversely, suppose that there exists a reduced circulation $\rho$ such that 
$0<\rho\le R$. Let $\gamma:=\rho-\rho^*$. Then $\gamma^*=-\gamma$, 
$\partial\gamma=2\partial\rho=0$, and $\gamma\ne0$. Since 
$\rho\le R$, taking transpose gives 
$\rho^*\le R$. Hence $$\gamma=\rho-\rho^*\le\rho\le\beta-\nu \quad \text{ and }\quad -\gamma=\rho^*-\rho\le\rho^*\le\beta-\nu.$$
Therefore $\nu\pm\gamma\le\beta$, and since $\beta\ge0$, this implies $(\nu\pm\gamma)_+\le\beta$. Together with $\gamma^*=-\gamma$ and $\partial\gamma=0$, we get 
$\nu\pm\gamma\in\G$, Thus $\nu$ is not an extreme point of $\G$, and consequently $[\mu]$ is not an 
extreme point of $\widetilde{\F}$. This proves the theorem.
\end{proof}

Since deleting a trivial flow from a flow $\mu$ does not change its boundary, the same reduction procedure applies to submodular flows with zero lower capacity. 
Fix a submodular network system $(J,\B,0,\beta,\varphi)$ and write
\[
    \F_\varphi(J,\B,\beta)
    :=
    \{\mu\in\mpBB:0\le\mu\le\beta,\ \partial\mu\le\varphi\} \quad \text{ and } \quad \widetilde{\F}_\varphi(J,\B,\beta)
    :=
    \F_\varphi(J,\B,\beta)/{\sim}.
\]

\begin{thm}\label{thm:reducedextremesubmodularflow}
Let $[\mu]\in\widetilde{\F}_\varphi(J,\B,\beta)$, and let
\[
    \psi:=\varphi-\partial\mu,
    \qquad
    \psi^c(X):=\psi(X^c),\quad X\in\B,
\]
and $R:=R([\mu];\beta)$.
Then $[\mu]$ is an extreme point of 
$\widetilde{\F}_\varphi(J,\B,\beta)$ if and only if the following two statements hold:
\begin{enumerate}
    \item\label{rcond1} 
        $\bmm(\kappa_{0,R})
        \cap
        \bmm(\psi)
        \cap
        \bmm(\psi^c)
        =
        \{0\};$
    \item\label{rcond2} there is no reduced circulation $\rho$ such that $0< \rho\le R.$
\end{enumerate}
\end{thm}
\begin{proof}
We may replace \(\mu\) by its reduced representative \(P(\mu)\), since
\(P(\mu)\sim\mu\) and \(\partial P(\mu)=\partial\mu\). Thus assume
\(\mu=P(\mu)\), and put \(\nu:=L(\mu)=\mu-\mu^*\). Then
\(\nu^*=-\nu\), \(\nu^+=\mu\), and
\[
    R=R([\mu];\beta)=(\beta-\nu)\wedge(\beta^*+\nu).
\]

Let
\[
    \G_\varphi:=
    \{\xi\in\mBB:\xi^*=-\xi,\ \xi^+\le\beta,\ \partial\xi\le 2\varphi\}.
\]
The map \([\theta]\mapsto L(\theta)\) identifies
\(\widetilde{\F}_\varphi(J,\B,\beta)\) affinely with \(\G_\varphi\). Indeed,
if \(\theta\in\F_\varphi(J,\B,\beta)\), then
\(L(\theta)^*=-L(\theta)\), \(L(\theta)^+=P(\theta)\le\theta\le\beta\), and
\(\partial L(\theta)=2\partial\theta\le2\varphi\). Conversely, if
\(\xi\in\G_\varphi\), then \(\xi=\xi^+-(\xi^+)^*\), so
\(L(\xi^+)=\xi\), while \(0\le\xi^+\le\beta\) and
\(\partial\xi^+=\frac12\partial\xi\le\varphi\). Hence
\(\xi^+\in\F_\varphi(J,\B,\beta)\).

Therefore \([\mu]\) is extreme in \(\widetilde{\F}_\varphi\) if and only if
\(\nu\) is extreme in \(\G_\varphi\). The latter fails if and only if there is
a nonzero \(\gamma\in\mBB\) with \(\nu\pm\gamma\in\G_\varphi\), equivalently,
\[
    \gamma^*=-\gamma,\qquad
    \gamma\le R,\qquad
    \partial\gamma\le2\psi,\qquad
    -\partial\gamma\le2\psi .
\]
Since \(\partial\gamma(J)=0\), the last two inequalities are equivalently
\[
    \frac12\partial\gamma\in\bmm(\psi)\cap\bmm(\psi^c).
\]

Write \(\rho:=\gamma^+\). Since \(\gamma^*=-\gamma\), we have
\(\gamma=\rho-\rho^*\), \(\rho\wedge\rho^*=0\), and
\(\partial\gamma=2\partial\rho\). Moreover, as in the proof of
\Cref{lem:extremereducedflow}, the condition \(\gamma\le R\) is equivalent to
\(0<\rho\le R\). Thus \([\mu]\) is not extreme if and only if there exists a
nonzero reduced flow \(\rho\) such that
\[
    0<\rho\le R,\qquad
    \partial\rho\in\bmm(\psi)\cap\bmm(\psi^c).
\]

By \Cref{partialinbmm},
\[
    \bmm(\kappa_{0,R})
    =
    \{\partial\rho:\rho\in\mBB,\ 0\le\rho\le R\}.
\]
Hence any such \(\rho\) with \(\partial\rho\ne0\) violates condition
\ref{rcond1}, while any such \(\rho\) with \(\partial\rho=0\) is a nonzero
reduced circulation and violates condition \ref{rcond2}. Conversely, if
condition \ref{rcond1} fails, choose
\(0\ne q\in\bmm(\kappa_{0,R})\cap\bmm(\psi)\cap\bmm(\psi^c)\). By
\Cref{partialinbmm}, there is \(0\le\theta\le R\) with \(\partial\theta=q\);
then \(\gamma:=\theta-\theta^*\) is a nonzero perturbation satisfying
\(\nu\pm\gamma\in\G_\varphi\). If condition \ref{rcond2} fails, the same
conclusion follows by taking \(\gamma:=\rho-\rho^*\) for a nonzero reduced
circulation \(0<\rho\le R\). Therefore \([\mu]\) is extreme exactly when both
conditions hold.
\end{proof}

\section{Applications}\label{sec5}

\subsection{Transshipment and extremality}

As an application, we state the transshipment problem on measurable spaces as follows.
Fix two standard Borel spaces $(J_1,\B_1)$ and $(J_2,\B_2)$. We always regard $J_1\sqcup J_2$ as a disjoint union, even if the underlying sets overlap, and its $\sigma$-algebra is
\[
    \B_1\oplus\B_2:=\{A_1\sqcup A_2:A_i\in\B_i\}.
\]
Given a supply measure $\alpha\in\M_+(\B_1)$, a demand measure $\beta\in\M_+(\B_2)$ with $\alpha(J_1)=\beta(J_2)$, and a capacity measure $\psi\in \M_+((\B_1\oplus \B_2)^2)$ supported on $J_1\times J_2$, the tuple $(J_1,\B_1,J_2,\B_2,\alpha,\beta,\psi)$ is called a \emphd{transshipment system}.

Define $\sigma_\alpha,\sigma_\beta\in\M_+(\B_1\oplus \B_2)$ by $$\sigma_\alpha(X)=\alpha(X\cap J_1), \qquad\sigma_\beta(X)=\beta(X\cap J_2).$$
A flow $\mu$ on $(J_1\sqcup J_2,\B_1\sqcup \B_2)$ is called a \emphd{feasible $\alpha$-$\beta$ transshipment} if 
$$0\leq \mu\leq \psi\quad\text{ and }\quad\mu^1=\sigma_\alpha,\;\; \mu^2=\sigma_\beta.$$

Since $\psi$ is supported on $J_1\times J_2$, this is equivalent to requiring $\mu$ to be a feasible $\sigma_\alpha$-$\sigma_\beta$ flow.
The set of all feasible $\alpha$-$\beta$ transshipment is denoted by $$\T=\T(J_1,\B_1,J_2,\B_2,\alpha,\beta,\psi).$$

\begin{rem}
$\T$ is the set of couplings of $\alpha,\beta$ constrained by $\psi$.
\end{rem}

If $\T\neq \emptyset$, then $\T$ is a nonempty weakly compact convex set, and one can consider its extreme points.
Applying~\Cref{lem:extremereducedflow}, every $\mu \in \T$ is supported on some $J_1\times J_2$, and hence automatically reduced and the quotient $\T/_\sim$ equals $\T$. Moreover,  for any $\mu\in \T$, the residual graph simplifies to 
$$R=R([\mu];\psi)=\mu\wedge (\psi-\mu)+\mu^*\wedge(\psi^*-\mu^*)$$
because $\mu,\psi$ are supported on $J_1\times J_2$, while $\mu^*,\psi^*$ are supported on $J_2\times J_1$.

As a corollary of~\Cref{lem:extremereducedflow}, we obtain the following characterization of extreme feasible flows.
\begin{coro}\label{Coro:acyclic}
    A measure $\mu$ is an extreme point of $\T$ if and only if there is no reduced circulation $\gamma$ such that $$0< \gamma\leq \mu\wedge (\psi-\mu)+\mu^*\wedge(\psi^*-\mu^*) .$$
\end{coro}
It remains to check whether $\T\neq \emptyset$. \cite[Proposition 4.15]{lovasz2021flows} gives a Hall-type characterization for the existence of feasible transshipments without the constraint measure $\psi$. Here we give a Hall-type characterization of the constrained version under certain assumptions on $\psi$.

\begin{prop}\label{prop:hall}
A measurable bipartite graph between \(J_1\) and \(J_2\) is a Borel set
\(E\subseteq J_1\times J_2\) such that for every \(S\in\mathcal B_1\),
\[N(S):=\{y\in J_2:\exists x\in S,\ (x,y)\in E\}\] belongs to \(\mathcal B_2\).
    Assume that \(\psi\) is supported on \(E\), and that for every
\(S_1\in\mathcal B_1\) and every \(S_2\in\mathcal B_2\) with
\(S_2\subseteq N(S_1)\), one has
\[
\psi(S_1\times S_2)\ge \beta(S_2).
\]
Then $\T\neq \emptyset$ if and only if $$\alpha(S)\leq \beta\big(N(S)\big)\quad \text{ for any $S\in \B_1$}.$$ 
\end{prop}

\begin{proof}
\emph{Necessity:}
 Let $\mu\in\T$. For $S\in\B_1$, since $\psi$ is supported on $E\subseteq J_1\times J_2$ and $\mu \le \psi$,
\[
    \mu(S\times(J_2\setminus N(S)))=0.
\]
Therefore
\[
    \alpha(S)=\mu(S\times J_2)=\mu(S\times N(S))\le\mu(J_1\times N(S))=\beta(N(S)).
\]

\noindent\emph{Sufficiency:}
Assume $\alpha(S)\le\beta(N(S))$ for every $S\in\B_1$. We apply \Cref{flowexistence} to the network system $(J_1\sqcup J_2,\B_1\oplus\B_2,0,\psi)$ with boundary $\sigma_\alpha-
\sigma_\beta$. Equivalently, applying the cut condition to complements, it is enough to prove that for every $S=S_1\sqcup S_2$ with $S_i\in\B_i$,
\[
    \psi(S\times S^c)\ge\sigma_\alpha(S)-\sigma_\beta(S)=\alpha(S_1)-\beta(S_2).
\]
Since $\psi$ is supported on $E\subseteq J_1\times J_2$,
\[
    \psi(S\times S^c)=\psi(S_1\times(J_2\setminus S_2))
    =\psi(S_1\times(N(S_1)\setminus S_2)).
\]
By the domination assumption on $\psi$,
\[
    \psi(S_1\times(N(S_1)\setminus S_2))
    \ge\beta(N(S_1)\setminus S_2)
    \ge\beta(N(S_1))-\beta(S_2)
    \ge\alpha(S_1)-\beta(S_2).
\]
Thus the cut condition holds and $\T\ne\emptyset$.
\end{proof}

\begin{rem}
    Without the condition that for every $S_1\in\B_1$ and every $S_2\in\B_2$ with $S_2\subseteq N(S_1)$, we have
    $$\psi(S_1\times S_2) \geq \beta(S_2),$$
    the theorem fails, as shown in \cite[Theorem 1.2]{kun2024measurablehall}. 
\end{rem}

\subsection{Fractional orientation of measurable graphs}
A classic application of discrete submodular flows is Frank's generalization of Nash-Williams' strong orientation theorem \cite{frank1996orientation}. It states that finding an edge orientation of an undirected graph where the in-degree of every subset is bounded below by a supermodular function is equivalent to finding a feasible submodular flow. Using our framework, we can generalize this result to standard Borel spaces as follows.

In this section we fix a standard Borel space $(J, \B)$ and a trivial flow $\Lambda \in \mpBB$ such that $\Lambda(\Delta_J) = 0$. Equivalently, $\Lambda=\Lambda^*$.
A \emphd{fractional orientation} of $\Lambda$ is a measure $\mu \in \mpBB$ such that 
\[    \mu + \mu^* = \Lambda.    \]
Note that this implies $0 \leq \mu \leq \Lambda$. For any subset $X \in \B$, the \emphd{inflow} into $X$ under the orientation $\mu$ is $\mu(X^c \times X)$. 

Suppose we are given a continuous supermodular set function $h \colon \B \to \R_{\geq 0}$ with $h(\emptyset) = h(J) = 0$, representing a minimum inflow demand required to maintain network connectivity. We seek an orientation $\mu$ satisfying the continuous degree constraint:
\begin{equation}\label{eq:orientation_demand}
    \mu(X^c \times X) \geq h(X) \textrm{ for all } X \in \B.
\end{equation}

By transforming this into a measurable submodular flow problem, we obtain the following continuous analogue of Frank's Orientation Theorem.

\begin{prop}\label{thm:measurable_orientation}
    Let $h$ be a continuous supermodular function on $\B$ with $h(\emptyset) = h(J) = 0$. There exists a fractional orientation $\mu$ of $\Lambda$ satisfying $\mu(X^c \times X) \geq h(X)$ for all $X \in \B$ if and only if
    \begin{equation}\label{eq:orientation_condition}
        h(X) \leq \Lambda(X \times X^c) \quad \text{for all } X \in \B.
    \end{equation}
\end{prop}

\begin{proof}
    Necessity is immediate, since 
    $$\mu(X^c\times X)\leq\Lambda(X^c\times X)= \Lambda(X\times X^c)$$
    for every $X\in\B$, where the last equality uses $\Lambda=\Lambda^*$.

    For sufficiency, we fix a Borel linear order $\prec$ on $J$, which exists on every standard Borel space, and define 
    $$\Omega:=\{(x,y)\in J\times J\;:\; x\prec y \},\qquad \mu_0:=\Lambda|_{\Omega}.$$
    Since $\Lambda$ is a trivial flow and $\Lambda(\Delta_J) = 0$, we clearly have $\mu_0 + \mu_0^* = \Lambda$.
    
    Any orientation $\mu$ can be generated by reversing a portion of the flow in $\mu_0$. Specifically, let $\theta \in \mpBB$ be a flow such that $0 \leq \theta \leq \mu_0$. We define the new orientation as
    \[
        \mu := \mu_0 - \theta + \theta^*.
    \]
    It is easy to verify that $\mu \geq 0$ and $\mu + \mu^* = \mu_0 + \mu_0^* = \Lambda$. 
    For this $\mu$, the inflow into $X$ is calculated as:
    \begin{align*}
        \mu(X^c \times X) &= \mu_0(X^c \times X) - \theta(X^c \times X) + \theta^*(X^c \times X) \\
        &= \mu_0(X^c \times X) - \theta(X^c \times X) + \theta(X \times X^c) \\
        &= \mu_0(X^c \times X) + \partial\theta(X).
    \end{align*}
    Thus, the inflow constraint \eqref{eq:orientation_demand} becomes $\partial\theta(X) \geq h(X) - \mu_0(X^c \times X)$. 
    
    Since $\partial\theta(X^c) = - \partial\theta(X)$, substituting $X$ with $X^c$ gives:
    \[
        \partial\theta(X) \leq \mu_0(X \times X^c) - h(X^c).
    \]
    Define the set function $$\varphi(X) := \mu_0(X \times X^c) - h(X^c).$$ Note that $\mu_0(X \times X^c)=\kappa_{0,\mu_0}(X)$, which is submodular, and $h$ is supermodular which implies $-h(X^c)$ is also submodular. Thus, $\varphi$ is a continuous submodular set function with $\varphi(\emptyset) = \varphi(J) = 0$.

    The problem is now reduced to finding a feasible submodular flow $\theta$ in the network system $(J, \B, 0, \mu_0, \varphi)$. By \Cref{submoexist}, such $\theta$ exists if and only if for all $X \in \B$,
    \[
        - \mu_0(X^c \times X) \leq \varphi(X) = \mu_0(X \times X^c) - h(X^c).
    \]
    Rearranging this inequality yields
    \[
        h(X^c) \leq \mu_0(X \times X^c) + \mu_0(X^c \times X) = (\mu_0 + \mu_0^*)(X \times X^c) = \Lambda(X \times X^c).
    \]
    Since $X$ ranges over $\B$ and $\Lambda$ is symmetric, this is exactly condition \eqref{eq:orientation_condition}, which completes the proof.   
\end{proof}

It remains to understand when the fractional orientation can be chosen to be an orientation. Since $\mu+
\mu^*=\Lambda$, an orientation corresponds to choosing one direction $\Lambda$-almost everywhere; equivalently,
\[
    \frac{\mathrm d\mu}{\mathrm d\Lambda}\in\{0,1\}
    \qquad \Lambda\text{-a.e.}
\]
This is also equivalent to $\mu\wedge\mu^*=0$. In the finite setting, integrality follows from the totally dual integral property of the submodular flow problem. In the measurable setting the corresponding non-fractional problem is more delicate. We leave it as a problem.
\begin{ques}
    Under what circumstances can the fractional orientation $\mu$ given by~\Cref{thm:measurable_orientation} be chosen such that $\mu \wedge \mu^* = 0$?
\end{ques}

\section*{Acknowledgements}

The second author would like to express his sincere gratitude to Krist\'of B\'erczi for the warm hospitality during his visit to the Alfr\'ed R\'enyi Institute of Mathematics. He is also grateful to L\'aszl\'o Lov\'asz for sharing with him the problem of extending extreme points of basic minorizing charges of coverage functions, which motivates part of this work.

\printbibliography

\end{document}